\documentclass[11pt]{amsart}

\usepackage{xspace}
\usepackage{graphicx}
\usepackage{longtable}
\usepackage{mathtools}
\usepackage{latexsym}
\usepackage{stmaryrd}
\usepackage{fullpage}
\usepackage{amsfonts,amsmath,amscd,amssymb}
\input{xy}
\xyoption{all}
\xyoption{curve}

\newcommand{\bG}{{\mathbb G}}
\newcommand{\bK}{{\mathbb K}}
\newcommand{\bF}{{\mathbb F}}
\newcommand{\bZ}{{\mathbb Z}}
\newcommand{\bN}{{\mathbb N}}

\newcommand{\fc}{{\mathfrak c}}
\newcommand{\fg}{{\mathfrak g}}
\newcommand{\fl}{{\mathfrak l}}
\newcommand{\fh}{{\mathfrak h}}
\newcommand{\fm}{{\mathfrak m}}
\newcommand{\fn}{{\mathfrak n}}
\newcommand{\fp}{{\mathfrak p}}

\newcommand{\fu}{{\mathfrak u}}

\newcommand{\fb}{{\mathfrak b}}

\newcommand{\GL}{{{\mbox{\rm GL}}}}

\newcommand{\ann}{{{\mbox{\rm ann}}}}

\newcommand{\Lie}{{{\mbox{\rm Lie}}}} 
\newcommand{\Maxspec}{{{\mbox{\rm Maxspec}}}}
\newcommand{\Maxi}{{{\mbox{\rm Max}}}}
\newcommand{\modu}{{{\mbox{\rm mod}}}}
\newcommand{\MaxIrr}{{{\mbox{\rm MaxIrr}}}}
\newcommand{\Spec}{{{\mbox{\rm Spec}}}} 
 
\newcommand{\Hom}{{{\mbox{\rm Hom}}}}
\newcommand{\Irr}{{{\mbox{\rm Irr}}}}
\newcommand{\Id}{{{\mbox{\rm Id}}}}

\newcommand{\PIDeg}{{{\mbox{\rm PI-degree}}}}

\newcommand{\End}{{{\mbox{\rm End}}}}
\newcommand{\Di}{{{\mbox{\rm Dist}}}}

\newcommand{\ve}{\ensuremath{\mathbf{e}}\xspace}

\newcommand{\vh}{\ensuremath{\mathbf{h}}\xspace}

\newcommand{\Upr}{U^{[r]}(G)}

\newcommand{\Dipri}{\Di^{+}_{p^r}(G)}

\newtheorem{theorem}{Theorem}[section]
\newtheorem{prop}[theorem]{Proposition}
\newtheorem{lemma}[theorem]{Lemma}
\newtheorem*{dfn}{Definition}
\newtheorem{cor}[theorem]{Corollary}

\newtheorem{rmk}{Remark}

\begin{document}
\title[Higher Deformations]{Higher Deformations of Lie Algebra Representations II}
\author{Matthew Westaway}
\email{M.P.Westaway@bham.ac.uk}
\address{School of Mathematics, University of Birmingham, Birmingham, B15 2TT, UK}
\date{June 2, 2020}
  
\begin{abstract}
	Steinberg's tensor product theorem shows that for semisimple algebraic groups the study of irreducible representations of higher Frobenius kernels reduces to the study of irreducible representations of the first Frobenius kernel. In the preceding paper in this series, deforming the distribution algebra of a higher Frobenius kernel yielded a family of deformations called higher reduced enveloping algebras. In this paper we prove that Steinberg decomposition can be similarly deformed, allowing us to reduce representation theoretic questions about these algebras to questions about reduced enveloping algebras. We use this to derive structural results about modules over these algebras. Separately, we also show that many of the results in the preceding paper hold without an assumption of reductivity.
\end{abstract}
\maketitle

\section{Introduction}

Let $G$ be a semisimple algebraic group over an algebraically closed field $\bK$ of characteristic $p>0$. We denote by $G_{r}$ the $r$-th Frobenius kernel of $G$. It was shown by Steinberg in 1963 \cite{Stein} that in order to understand the irreducible $G_{r}$-modules for $r\geq 1$, it is sufficient to understand the irreducible $G_{1}$-modules. This result can be interpreted in the following way: considering irreducible modules only up to isomorphism, there is a bijection
$$\Psi_0:\Irr(\Di(G_{r+1}))\to \Irr(\Di(G_{r}))\times\Irr(\Di(G_{1})),$$
recalling here that the category of $G_{r}$-modules is equivalent to the category of $\Di(G_{r})$-modules, where $\Di(G_{r})$ is the distribution algebra of $G_{r}$. In particular, this bijection sends the irreducible $\Di(G_{r+1})$-module $L_{r+1}(\lambda + \mu p^r)$, where $\lambda\in X_r$ and $\mu\in X_1$, to the pair $(L_r(\lambda),L_1(\mu))$. Here, $X_r$ is the set of dominant weights $\lambda$ of some maximal torus $T$ of $G$ which satisfy that $0\leq\langle\lambda,\alpha^\nu \rangle< p^r$ for all simple coroots $\alpha^\nu$ of $G$ with respect to $T$.

In the previous paper in this series \cite{West} we constructed, for each $r\in\bN$, a higher universal enveloping algebra $U^{[r]}(G)$ and, for each $\chi\in\Lie(G)^{*}=\fg^{*}$, a reduced higher universal enveloping algebra $U_\chi^{[r]}(G)$, with the key property that $U^{[r]}_0(G)\cong\Di(G_{r+1})$. Every irreducible $U^{[r]}(G)$-module is a $U_\chi^{[r]}(G)$-module for some $\chi$, and in \cite{West} it was shown that, under certain restrictions, there is a well-defined map
$$\Psi_{\chi}:\Irr(U^{[r]}_\chi(G))\to \Irr(\Di(G_{r}))\times\Irr(U_\chi(\fg))$$
which, when $\chi=0$, gives Steinberg decomposition. 

In this paper we remove the restrictions and furthermore show that this map is always a bijection (Theorem~\ref{equiv}, Corollary~\ref{chibij}). This then allows us to derive various structural results about the irreducible $U^{[r]}_\chi(G)$-modules. In particular, given an irreducible $\Di(G_{r})$-module $P$ one can construct teenage Verma modules $Z_\chi^r(P,\lambda)$ which behave as the baby Verma modules $Z_\chi(\lambda)$ do (Proposition~\ref{tVm}). This allows us to classify all irreducible $U^{[r]}_\chi(G)$-modules when $\chi$ is regular in Theorem~\ref{regular}.

The main techniques which allow us to prove these results come from the work of Schneider and Witherspoon on Clifford theory for Hopf algebras. In fact, the Hopf algebraic approach also allows us to reprove many of the results from \cite{West} for affine algebraic groups which are not necessarily reductive. In particular, we show that $U^{[r]}(G)$ is a crossed product of $\Di(G_{r})$ with $U(\fg)^{(r)}$ in Proposition~\ref{cleft}, and that $U^{[r]}(G)$ has a Poincar\'{e}-Birkhoff-Witt basis in Corollary~\ref{basis}. This is the content of Section~\ref{s1}.

It is in Section~\ref{s2} where we study the representation theory of the higher universal enveloping algebras. Specifically, in Subsection~\ref{s2.1} we prove the main result - that the map $\Psi_\chi$ mentioned above is well-defined and a bijection. Then, in Subsection~\ref{s2.2} we construct the teenage Verma modules $Z_\chi^r(P,\lambda)$ and prove some preliminary results about them. Finally, in Subsection~\ref{s2.3}, we see some consequences of the results proved in the previous two subsections.

We conclude in Section~\ref{s3} with a discussion of the Azumaya locus of the algebras $U^{[r]}(G)$. In particular, we start by discussing the Azumaya locus of a not-necessarily-prime algebra $R$ with affine centre $Z$, over which $R$ is module-finite. The reader should note that the prime case has previously been studied by Brown and Goodearl in \cite{BGo}. We see that, under certain conditions, the Azumaya locus coincides with the \emph{pseudo-Azumaya locus}, which is defined in Subsection~\ref{s3.1} and uses the representation theory of $R$. In Subsection~\ref{s3.2} we see how the pseudo-Azumaya locus of the algebra $U^{[r]}(G)$ connects to the Azumaya locus of the corresponding $U(\fg)$.

This work was completed while I was a postgraduate student at the University of Warwick. I would like to thank my PhD supervisors Dmitriy Rumynin and Inna Capdeboscq for their continued assistance with this project. I would also like to thank Lewis Topley for some useful discussions regarding this subject. Finally, I want to thank Alexander Premet, Adam Thomas, and the referee for their useful comments which have helped improve the paper. I was supported during this research by a PhD studentship from the Engineering and Physical Sciences Research Council.
\section{Notation}

Let $A$ be an associative $\bK$-algebra, where $\bK$ is an algebraically closed field of characteristic $p>0$. From now on, we shall write $\Irr(A)$ for the category of irreducible left $A$-modules. In all instances in this paper, elements of the set $\Irr(A)$ shall be finite-dimensional. Given a vector space $V$ we shall write $V^{(r)}$ for the vector space with the same underlying abelian group as $V$ but whose scalar multiplication is given by the map $\bK\times V\to \bK\times V\to V$ which is a composition of the map $(\lambda,v)\mapsto (\lambda^{p^{-r}},v)$ with the scalar multiplication map on $V$. In particular, we denote by $A^{(r)}$ the algebra with underlying ring $A$ but underlying vector space $A^{(r)}$.

When $G$ is a reductive algebraic group over an algebraically closed field $\bK$ of characteristic $p>0$, we assign a maximal torus $T$ and Borel subgroup $B$ such that $T\subset B\subset G$. We also let $\Phi$ denote the root system of $G$ with respect to $T$, let $\Pi$ be a choice of simple roots, and let $\Phi^{+}$ be the corresponding set of positive roots. We further define $\fg=\Lie(G)$, $\fb=\Lie(B)$ and $\fh=\Lie(T)$. For $\alpha\in\Phi$ we define $\fg_{\alpha}$ to be the corresponding root space of $\fg$, and we set $\fn^{+}=\bigoplus_{\alpha\in\Phi^{+}}\fg_{\alpha}$ and $\fn^{-}=\bigoplus_{\alpha\in\Phi^{+}}\fg_{-\alpha}$.

The character group of $T$ will be denoted $X(T)=\Hom(T,\bG_m)$ and the cocharacter group of $T$ will be denoted by $Y(T)=\Hom(\bG_m,T)$. We shall denote by $<\cdot,\cdot>:X(T)\times Y(T)\to \bZ$ the standard bilinear form as in \cite[II.1.3]{Jan3}.

The Lie algebra $\fg$ has basis consisting of $\ve_{\alpha}$ for $\alpha\in\Phi$ and $\vh_t$ for $1\leq t\leq d$, where $d=\dim(\fh)$, as in \cite[II.1.11]{Jan3}.

\section{Poincar\'{e}-Birkhoff-Witt for higher universal enveloping algebras}
\label{s1}

Let $G$ be an algebraic group over the algebraically closed field $\bK$, with coordinate algebra $\bK[G]$. Let us recall the construction of the distribution algebra of $G$ and of the higher universal enveloping algebras of $G$. 

For $n\in\bN$, we define the vector space $\Di_n(G)$ to consist of all linear maps $\delta:\bK[G]\to\bK$ which vanish on $I^{n+1}$, where $I$ is the augmentation ideal of $\bK[G]$. We further define $\Di_n^{+}(G)$ to be the subspace of all $\delta\in\Di_n(G)$ with $\delta(1)=0$. The distribution algebra of $G$ is then defined to be the algebra
$$\Di(G)=\bigcup_{n\in\bN}\Di_n(G),$$
with multiplication defined as follows: if $\delta\in\Di_n(G)$, $\mu\in\Di_m(G)$, then $\delta\mu$ is the map
$$\bK[G]\xrightarrow{\Delta}\bK[G]\otimes\bK[G]\xrightarrow{\delta\otimes\mu} \bK\otimes\bK\xrightarrow{\sim}\bK,$$
where $\Delta$ is the comultiplication map on $\bK[G]$. In particular, one can show that $\delta\mu\in\Di_{n+m}(G)$ and $[\delta,\mu]\in\Di_{n+m-1}(G)$. The algebra has the structure of a cocommutative Hopf algebra.

For $r\in\bN$, we can define (as in \cite{West}) the $r$-th higher universal enveloping algebra $U^{[r]}(G)$ as follows:
$$U^{[r]}(G)\coloneqq \frac{T(\Di^{+}_{p^{r+1}-1}(G))}{Q_r},$$
where $Q_r$ is the ideal generated by the two relations 

(i) $\delta\otimes \mu=\delta\mu$ if $\delta\in \Di^{+}_i(G)$, $\mu\in \Di^{+}_j(G)$ with $i+j<p^{r+1}$; and,

(ii) $\delta\otimes \mu-\mu\otimes \delta = [\delta,\mu]$ if $\delta\in \Di^+_i(G)$, $\mu\in \Di^+_j(G)$ with $i+j\leq p^{r+1}$,

and $T(\Di^{+}_{p^{r+1}-1}(G))$ is the tensor algebra of $\Di^{+}_{p^{r+1}-1}(G)$.
This algebra also has the structure of a cocommutative Hopf algebra.

In order to construct a Poincar\'{e}-Birkhoff-Witt basis of $U^{[r]}(G)$, we need to use a couple of Hopf algebraic notions. For a Hopf algebra $H$, we define the set of primitive elements $P(H)\coloneqq \{x\in H\,\vert\,\Delta(x)=x\otimes 1 + 1\otimes x\}$, and the set of group-like elements $G(H)\coloneqq \{x\in H\,\vert\,\Delta(x)=x\otimes x\}$. Given an element $x\in P(H)$, a sequence $x^{(0)},x^{(1)},x^{(2)},\ldots,x^{(k)}\in H$ is said to be a sequence of divided powers of $x$ if 
(i) $x^{(0)}=1$;
(ii) $x^{(1)}=x$; and,
(iii) $\Delta(x^{(l)})=\sum_{i=0}^l x^{(i)}\otimes x^{(l-i)}$ for all $l\geq 0$.

Suppose that $x_1,\ldots,x_n$ is a basis for the Lie algebra $\fg=\Lie(G)$. For each $1\leq i\leq n$, there exists an infinite sequence of divided powers $x_i^{(0)},x_i^{(1)},x_i^{(2)},\ldots$ of $x_i$ in the cocommutative Hopf algebra $\Di(G)$. It is well-known (see \cite{Sweed}) that the distribution algebra $\Di(G_{r})$ has basis
$$\{x_1^{(a_1)}x_2^{(a_2)}\ldots x_n^{(a_n)}\,\vert\, 0\leq a_i<p^r\,\mbox{for all}\,1\leq i\leq n\},$$
while the vector space $\Di_{k}(G)$ has basis
$$\{x_1^{(a_1)}x_2^{(a_2)}\ldots x_n^{(a_n)}\,\vert\, \sum_{i=1}^n a_i \leq k\}.$$

One can also observe that $x_i^{(k)}\in\Di_k(G)$ for all $1\leq i\leq n$ and $k\in\bN$.

In particular, there is an inclusion of vector spaces $\Di^{+}_{p^{r}-1}(G)\hookrightarrow \Di(G_{r})\subset \Di(G)$ which clearly satisfies the necessary conditions to employ the universal property of $U^{[r-1]}(G)$ and obtain an algebra homomorphism
$$\pi_{r-1}:U^{[r-1]}(G)\to\Di(G_{r}).$$
From the basis description of $\Di(G_{r})$ above, this map is surjective. 

It is straightforward to see that for $\delta\in\Di^{+}_{p^{r-1}}(G)$ the equality $\pi_{r-1}(\delta)^p=\pi_{r-1}(\delta^p)$ holds. Hence, letting $R_{r-1}$ be the ideal of $U^{[r-1]}(G)$ generated by $\delta^{\otimes p}-\delta^p$ for $\delta\in\Di^{+}_{p^{r-1}}(G)$, there is a surjective algebra homomorphism
$$\overline{\pi_{r-1}}:U^{[r-1]}(G)/R_{r-1}\twoheadrightarrow \Di(G_{r}).$$

\begin{lemma}\label{gens}
	The algebra $U^{[r-1]}(G)$ is spanned by the set 
	$$\{x_1^{(a_1)}\otimes (x_1^{(p^{r-1})})^{\otimes b_1}\otimes x_2^{(a_2)}\otimes (x_2^{(p^{r-1})})^{\otimes b_2}\otimes\ldots\otimes x_n^{(a_n)}\otimes (x_n^{(p^{r-1})})^{\otimes b_n}\,\vert\;0\leq a_i<p^{r-1},\, b_i\geq 0, \,\,1\leq i\leq n\}.$$
\end{lemma}

\begin{proof}
	That these elements generate $U^{[r-1]}(G)$ is obvious from the given basis of $\Di_{p^{r}-1}(G)$. Hence, using a filtration argument, all that remains is to make the following observations:
	
	(i) For $1\leq i\leq n$, if $0\leq s,t\leq p^{r-1}$, then $x_i^{(s)}\otimes x_i^{(t)}-\binom{s+t}{s}x_i^{(s+t)}$ lies in the $\bK$-span of the set $$
	\{x_1^{(a_1)}\otimes x_2^{(a_2)}\otimes\ldots\otimes x_n^{(a_n)}\,\vert \,\,0\leq a_j<p^{r-1},\,\,1\leq j\leq n, \;\mbox{and}\,\sum_{j=1}^{n}a_j<s+t\}.
	$$
	Note here that $\binom{s+t}{s}=0$ if $s+t\geq p^{r-1}$ and $s,t<p^{r-1}$.
	
	(ii) For $0\leq s,t\leq p^{r-1}$ and $1\leq i\leq j\leq n$, the commutator $x_j^{(t)}\otimes x_i^{(s)}-x_i^{(s)}\otimes x_j^{(t)}$ lies in the $\bK$-span of the set
	$$\left\{\begin{array}{l}\quad
	x_1^{(a_1)}\otimes (x_1^{(p^{r-1})})^{\otimes b_{ 1}}\otimes x_2^{(a_2)}\otimes (x_2^{(p^{r-1})})^{\otimes b_2}\otimes\ldots\otimes x_n^{(a_n)}\otimes (x_n^{(p^{r-1})})^{\otimes b_n}\\\mbox{with} \;\;0\leq a_k<p^{r-1},\,\, b_k\geq 0,\,\,1\leq k\leq n,\;\mbox{and} \sum_{k=1}^{n}(a_k + b_kp^{r-1}) <s+t
	\end{array}\right\}.$$
	
	These observations follow from the defining relations of $U^{[r-1]}(G)$ and calculations with the divided power basis of $\Di(G_{r})=\bK[G_r]^{*}$.
	
\end{proof}

\begin{cor}
	The algebra $U^{[r-1]}(G)/R_{r-1}$ is spanned by the set 
	$$\{x_1^{(a_1)}\otimes (x_1^{(p^{r-1})})^{\otimes b_1}\otimes x_2^{(a_2)}\otimes (x_2^{(p^{r-1})})^{\otimes b_2}\otimes\ldots\otimes x_n^{(a_n)}\otimes (x_n^{(p^{r-1})})^{\otimes b_n}\,\vert\;0\leq a_i<p^{r-1},\, 0\leq b_i<p,\,1\leq i\leq n\}.$$	
\end{cor}
\begin{proof}
	This follows from the above lemma since, for $\delta\in\Di_{p^{r-1}}(G)$, $\delta^p\in\Di_{p^{r}-1}(G)$ by Lemma 3.2.1 in \cite{West}.
\end{proof}

Hence, $\dim(U^{[r-1]}(G)/R_{r-1})\leq p^{r\dim(\fg)}$. However, we know that $U^{[r-1]}(G)/R_{r-1}$ surjects onto $\Di(G_{r})$, which has dimension $p^{r\dim(\fg)}$. Thus, $U^{[r-1]}(G)/R_{r-1}\cong \Di(G_{r})$. 

In particular, the universal property of the algebra $U^{[r-1]}(G)/R_{r-1}$ gives an algebra homomorphism $\Di(G_{r})\to U^{[r]}(G)$. Composing with $\pi_r$ then gives an algebra homomorphism $\Di(G_{r})\to\Di(G_{r+1})$ which, by considering the effect on the basis, is clearly injective. Hence, there is an inclusion $\Di(G_{r})\hookrightarrow U^{[r]}(G)$ of algebras.

The above results show that $\Di(G_{r})$ is a Hopf subalgebra of $U^{[r]}(G)$, since the coalgebra structure on $U^{[r]}(G)$ is extended from the coalgebra structure on $\Di_{p^{r+1}-1}(G)\subseteq\Di(G_{r})$ using the universal property given in Proposition 3.1.1 in \cite{West}, and similarly for the antipode. In fact, the given bases of $\Di(G_{r})$ and of $\Di_k(G)$ show that, as in Lemma 7.1.1(1) in \cite{West}, $\Di(G_{r})$ is normal in $U^{[r]}(G)$.

More generally, the results of Section 4 in \cite{West} all hold for an arbitrary affine algebraic group $G$ -- with one notable difference. Namely, we may no longer assume that $G$ has an $\bF_p$-form, and so we must use the standard Frobenius morphism rather than the geometric Frobenius morphism throughout. The reader can check that the only meaningful change this induces is to turn $\Upsilon_{r,s}$ into a Hopf algebra homomorphism from $U^{[r]}(G)$ to $U^{[r-s]}(G)^{(s)}$ instead of $U^{[r-s]}(G)$. Other than this, the only place in which the reductivity of $G$ is used in that section is to show that $\Upsilon_{r,s}$ is surjective, which now follows from Lemma~\ref{gens}. Hence, the whole of Lemma 7.1.1 in \cite{West} holds for an arbitrary affine algebraic group. 

In particular, $\Di(G_{r})\subset U^{[r]}(G)$ is a $U(\fg)^{(r)}$-Galois extension with $\Di(G_{r})=U^{[r]}(G)^{co U(\fg)^{(r)}}$.

\begin{prop}\label{cleft}
	The $U(\fg)^{(r)}$-extension $\Di(G_{r})\subset U^{[r]}(G)$ is $U(\fg)^{(r)}$-cleft.
\end{prop}

\begin{proof}
	We need to show that there is a convolution-invertible right $U(\fg)^{(r)}$-comodule map $\gamma:U(\fg)^{(r)}\to U^{[r]}(G)$. Since $U(\fg)^{(r)}$ has basis $\{ x_1^{a_1}x_2^{a_2}\ldots x_n^{a_n}\,\vert\; a_i\geq 0,\,1\leq i\leq n\}$, we simply need to define $\gamma(x_1^{a_1}x_2^{a_2}\ldots x_n^{a_n})$ for all $a_1,a_2,\ldots,a_n\geq0$.
	
	As such, we define $$\gamma(x_1^{a_1}x_2^{a_2}\ldots x_n^{a_n})=(x_1^{(p^r)})^{\otimes a_1}\otimes (x_2^{(p^r)})^{\otimes a_2}\otimes\ldots\otimes (x_n^{(p^r)})^{\otimes a_n}\in U^{[r]}(G)$$
	for all $a_1,a_2,\ldots,a_n\geq 0$.
	
	To show that $\gamma$ is a $U(\fg)^{(r)}$-comodule map we need to show that, for $y\in U(\fg)^{(r)}$,
	$$\sum\gamma(y)_{(1)}\otimes \overline{\gamma(y)_{(2)}}=\sum \gamma(y_{(1)})\otimes y_{(2)}$$
	where we use Sweedler's $\Sigma$-notation and we write $\overline{\gamma(y)_{(2)}}$ for $\Upsilon_{r,r}(\gamma(y)_{(2)})$.
	
	It is enough to show this for basis elements. Note that, if $y=x_1^{a_1}x_2^{a_2}\ldots x_n^{a_n}$ with $a_1,a_2,\ldots,a_n\geq 0$, then
	\begin{align*}
		\Delta(y)=(x_1\otimes 1 + 1\otimes x_1)^{a_1}(x_2\otimes 1 + 1\otimes x_2)^{a_2}\ldots (x_n\otimes 1 + 1\otimes x_n)^{a_n}\\
		= \sum_{b_i+c_i=a_i}\binom{a_1}{b_1}\binom{a_2}{b_2}\ldots\binom{a_n}{b_n} x_1^{b_1}x_2^{b_2}\ldots x_n^{b_n}\otimes x_1^{c_1}x_2^{c_2}\ldots x_n^{c_n}.
	\end{align*}
	
	Furthermore, writing $\Delta_{U(\fg)^{(r)}}$ for the $U(\fg)^{(r)}$-comodule map of the comodule $U^{[r]}(G)$,
	\begin{multline*}
		\Delta_{U(\fg)^{(r)}}((x_1^{(p^r)})^{\otimes a_1}\otimes (x_2^{(p^r)})^{\otimes a_2}\otimes\ldots\otimes (x_n^{(p^r)})^{\otimes a_n})\\
		=\Delta_{U(\fg)^{(r)}}(x_1^{(p^r)})^{\otimes a_1}\otimes \Delta_{U(\fg)^{(r)}}(x_2^{(p^r)})^{\otimes a_2}\otimes\ldots\otimes \Delta_{U(\fg)^{(r)}}(x_n^{(p^r)})^{\otimes a_n},
	\end{multline*}
	while, for any $1\leq i\leq n$,	
	$$\Delta_{U(\fg)^{(r)}}(x_i^{(p^r)})=\sum_{j=0}^{p^r} x_i^{(j)}\otimes\overline{x_i^{(p^r-j)}}=x_i^{(p^r)}\otimes 1 + 1\otimes x_i$$
	since $\overline{x_i^{(s)}}=0$ for all $0<s<p^r$.
	
	Hence, $\sum\gamma(y)_{(1)}\otimes \overline{\gamma(y)_{(2)}}$ equals
	\begin{multline*}
		\sum_{b_i+c_i=a_i}\binom{a_1}{b_1}\binom{a_2}{b_2}\ldots\binom{a_n}{b_n} ((x_1^{(p^r)})^{\otimes b_1}\otimes (x_2^{(p^r)})^{\otimes b_2}\otimes\ldots\otimes (x_n^{(p^r)})^{\otimes b_n}) \otimes (x_1^{c_1}x_2^{c_2}\ldots x_n^{c_n})
	\end{multline*}	
	
	and $\sum\gamma(y_{(1)})\otimes y_{(2)}$ equals
	\begin{multline*}
		\sum_{b_i+c_i=a_i}\binom{a_1}{b_1}\binom{a_2}{b_2}\ldots\binom{a_n}{b_n} ((x_1^{(p^r)})^{\otimes b_1}\otimes (x_2^{(p^r)})^{\otimes b_2}\otimes\ldots\otimes (x_n^{(p^r)})^{\otimes b_n}) \otimes (x_1^{c_1}x_2^{c_2}\ldots x_n^{c_n}).
	\end{multline*}
	
	Thus, $\gamma$ is a $U(\fg)^{(r)}$-comodule map. Furthermore, $\gamma$ is convolution-invertible (with convolution inverse $S\gamma$), since $U^{[r]}(G)$ is a Hopf algebra.
\end{proof}

By Theorem 8.2.4 in \cite{Mont}, $\Di(G_{r})\subset U^{[r]}(G)$ has the normal basis property. Hence, $U^{[r]}(G)\cong \Di(G_{r})\otimes U(\fg)^{(r)}$ as left $\Di(G_{r})$-modules and right $U(\fg)^{(r)}$-comodules. In particular, Corollary 8.2.5 in \cite{Mont} shows that 
$$U^{[r]}(G)\cong\Di(G_{r})\#_{\sigma} U(\fg)^{(r)},$$
a crossed product of $\Di(G_{r})$ with $U(\fg)^{(r)}$.

\begin{cor}\label{basis}
	$U^{[r]}(G)$ has basis 
	$$\{x_1^{(a_1)}x_2^{(a_2)}\ldots x_n^{(a_n)}(x_1^{(p^r)})^{b_1}(x_2^{(p^r)})^{b_2}\ldots (x_n^{(p^r)})^{b_n}\,\vert\;0\leq a_i<p^{r},\, 0\leq b_i,\, 1\leq i\leq n\}.$$
\end{cor}

[Note that in this corollary we suppress the $\otimes$-symbol when we write the multiplication in $U^{[r]}(G)$. We shall do similarly throughout this paper when no confusion is likely].

Now that we know a basis for $U^{[r]}(G)$, we can obtain the following corollary. The idea for this proof is due to Lewis Topley.
\begin{cor}\label{central}
	Let $G$ be an affine algebraic group. For $\delta\in\Di_{p^r}^{+}(G)$, $\delta^{\otimes p} -\delta^p$ is central in $U^{[r]}(G)$.
\end{cor}

\begin{proof}
	If $G$ is an affine algebraic group, then there is an inclusion $\Di(G)\subseteq \Di(\GL_m)$ for some $m\in\bN$, which restricts to an inclusion $\Di_k(G)\subseteq\Di_k(\GL_m)$ for all $k\in\bN$. In particular, the inclusion $\Di_{p^{r+1}-1}^{+}(G)\hookrightarrow\Di_{p^{r+1}-1}^{+}(\GL_m)\hookrightarrow U^{[r]}(\GL_m)$ induces, by the universal property, an algebra homomorphism
	$$\iota:U^{[r]}(G)\to U^{[r]}(\GL_m).$$
	Let $x_1,\ldots,x_n$ be a basis of $\fg=\Lie(G)$. This can be extended to a basis $x_1\ldots,x_{m^2}$ of $\fg\fl_m=\Lie(\GL_m)$. 
	
	The map $\iota$ sends $$x_1^{(a_1)}x_2^{(a_2)}\ldots x_n^{(a_n)}(x_1^{(p^r)})^{b_1}(x_2^{(p^r)})^{b_2}\ldots (x_n^{(p^r)})^{b_n}\in U^{[r]}(G)$$ to $$x_1^{(a_1)}x_2^{(a_2)}\ldots x_n^{(a_n)}(x_1^{(p^r)})^{b_1}(x_2^{(p^r)})^{b_2}\ldots (x_n^{(p^r)})^{b_n}\in U^{[r]}(\GL_m).$$ Hence, by Corollary~\ref{basis}, $\iota$ is injective.
	
	In particular, there is an inclusion $\iota:U^{[r]}(G)\hookrightarrow U^{[r]}(\GL_m)$. Now, for $\delta\in\Di_{p^r}^{+}(G)$, $\iota(\delta)^{\otimes p} - \iota(\delta)^p$ is central in $U^{[r]}(\GL_m)$ by \cite{West}, since $\GL_m$ is reductive.
	
	Hence, $\delta^{\otimes p} - \delta^p$ is central in $U^{[r]}(G)$.
\end{proof}

We can now proceed as in Section 3.4 in \cite{West} to obtain a number of corollaries for an arbitrary algebraic group $G$. Let $Z^{[r]}_p$ be the central subalgebra of $U^{[r]}(G)$ generated by all $\delta^{\otimes p} - \delta^p$ for $\delta\in\Di_{p^r}^{+}(G)$.

\begin{cor}
	The algebra $Z^{[r]}_p$ is generated by the elements $(x_i^{(p^r)})^{\otimes p}-(x_i^{(p^r)})^{p}$ for $i=1,\ldots,n$. Furthermore, these elements are algebraically independent.
\end{cor}

\begin{cor}
	As a $Z^{[r]}_p$-module, $U^{[r]}(G)$ is free with basis $$\{x_1^{(a_1)}x_2^{(a_2)}\ldots x_n^{(a_n)}\,\vert\,0\leq a_1,\ldots,a_n<p^{r+1}\,\}.$$
\end{cor}

\begin{cor}\label{fingen}
	The centre $Z^{[r]}(G)\coloneqq Z(\Upr)$ of $\Upr$ is a finitely generated algebra over $\bK$. As a $Z(\Upr)$-module, $\Upr$ is finitely generated.
\end{cor}

\begin{cor}\label{findim}
	Let $M$ be an irreducible $\Upr$-module. Then $M$ is finite-dimensional, of dimension less than or equal to $p^{(r+1)\dim(\fg)}$.
\end{cor}

Similarly, the requirement in Section 5.1 of \cite{West} that $G$ be reductive can be removed. In particular, for an arbitrary affine algebraic group $G$ and $\chi\in(\fg^{*})^{(r)}$ we can define the algebra
$$U_\chi^{[r]}(G)\coloneqq \frac{U^{[r]}(G)}{\langle \delta^{\otimes p} - \delta^p-\chi(\delta)^p\,\vert\,\delta\in\Dipri\rangle}.$$ Recall here that $\chi$ extends to $\Dipri$ through the map $\Upsilon_{r,r}:U^{[r]}(G)\to U(\fg)^{(r)}$ defined in Section 4 in \cite{West} -- the reader should note that this map is obtained from the Frobenius map $\Di(G)\to\Di(G^{(r)})$. We saw earlier that all the properties of this map given in \cite{West} for reductive groups also hold for affine algebraic groups.  We then obtain the following corollaries.

\begin{cor}
	Every irreducible $U^{[r]}(G)$-module is a $U^{[r]}_\chi(G)$-module for some $\chi\in(\fg^{*})^{(r)}$.
\end{cor}

\begin{cor}\label{orbit}
	Given $\chi\in(\fg^{*})^{(r)}$ and $g\in G$, there is an isomorphism $U^{[r]}_\chi(G)\cong U^{[r]}_{g\cdot\chi}(G)$, where $G$ is acting on $(\fg^{*})^{(r)}$ through the coadjoint action pre-composed with the $r$-th Frobenius morphism.
\end{cor}

Furthermore, it is a straightforward consequence of Corollary~\ref{basis} that $U^{[r]}_\chi(G)$ has basis
$$\{x_1^{(a_1)}x_2^{(a_2)}\ldots x_n^{(a_n)}\,\vert\;0\leq a_i<p^{r+1}\;\mbox{for all}\; 1\leq i\leq n\}.$$
Hence, $U^{[r]}_\chi(G)$ is a finite-dimensional algebra of dimension $p^{(r+1)\dim(\fg)}$.	

\section{Representation Theory of $U^{[r]}(G)$}
\label{s2}

\subsection{Steinberg decomposition}
\label{s2.1}

For the rest of this paper we assume that $G$ is a connected reductive algebraic group over $\bK$. We shall furthermore assume that the quotient group $X(T)/p^rX(T)$ has a system of representatives $X_r'(T)$ which lies inside $X_r(T)$. Recall that the definition of $X_r(T)$ is
$$X_r(T)\coloneqq\{\lambda\in X(T)\vert\,0\leq\langle \lambda,\alpha^{\nu}\rangle< p^r \;\mbox{for all}\;\alpha\in\Pi\}.$$
This assumption holds if, for example, $G$ is semisimple and simply-connected. The reader should consult \cite[II.3.16]{Jan3} to see how Steinberg's tensor product theorem works for reductive algebraic groups satisfying this assumption. In particular, this assumption guarantees that every irreducible $\Di(G_r)$-module extends to a $\Di(G_{r+1})$-module (and hence to a $U^{[r]}(G)$-module).

Observe that in this section our algebraic group $G$ has an $\bF_p$-form, and so we shall generally use the geometric Frobenius morphism rather than the standard Frobenius morphism. In particular, the homomorphisms $\Upsilon_{r,s}$ map from $U^{[r]}(G)$ to $U^{[r-s]}(G)$ without requiring a twist of the $\bK$-structure.

In \cite{West}, it is shown by two different methods that every irreducible $U^{[r]}(G)$-module $M$ is isomorphic as $U^{[r]}(G)$-modules (and hence $\Di(G_r)$-modules) to $P\otimes \Hom_{\tiny G_{r}}(P,M)$ for some unique (up to isomorphism) irreducible $P\in\Irr(\Di(G_{r}))$. The first method uses the fact that each irreducible $\Di(G_r)$-module $P$ can be extended to a $U^{[r]}(G)$-module, together with the Hopf algebra structure of $U^{[r]}(G)$, to equip $\Hom_{\tiny G_{r}}(P,M)$ with the structure of a $U(\fg)$-module and $P\otimes \Hom_{\tiny G_{r}}(P,M)$ with the structure of a left $U^{[r]}(G)$-module. The second method introduces the algebra $$E\coloneqq\End_{U^{[r]}(G)}(U^{[r]}(G)\otimes_{\tiny \Di(G_{r})} P)^{op},$$ and shows that $\Hom_{\tiny G_{r}}(P,M)$ has the structure of a left $E$-module. Theorem 7.1.3 in \cite{West} then gives a $U^{[r]}(G)$-module structure to $P\otimes \Hom_{\tiny G_{r}}(P,M)$, and Theorem 7.1.4 shows that it is compatible with the module structure on $M$.

In understanding the structure of $E$, the following lemma was proved in \cite{West} as Lemma 7.1.5. We repeat the lemma here, since we are now in a position to explain the isomorphism in more detail.

\begin{lemma}\label{End}
	Let $P\in\Irr(\Di(G_{r})$ and $E=\End_{U^{[r]}(G)}(U^{[r]}(G)\otimes_{\tiny \Di(G_{r})} P)^{op}$. Then $E\cong U(\fg)$.
\end{lemma}
%

\begin{rmk}\label{isom}
	We can describe this isomorphism a little more explicitly. The isomorphism $U(\fg)\cong \bK\# U(\fg)$ sends $z\in U(\fg)$ to $1\# z\in \bK\# U(\fg)$. We now need to consider the isomorphism $\bK\# U(\fg)\cong E$ from Schneider \cite{Sch}.
	
	Note that the stability of the $\Di(G_{r})$-module $P$ comes immediately from the fact that $P$ can be extended to a $U^{[r]}(G)$-module, by Remark 3.2.3 in \cite{Sch}. Let $q:U^{[r]}(G)\otimes_{D} P\to P$ be the $\Di(G_{r})$-linear map defining this $U^{[r]}(G)$-module structure, denoting the algebra $\Di(G_{r})$ by $D$ here and throughout this paper. By Theorem 3.6 in \cite{Sch}, there is a right $U(\fg)$-collinear map $J':U(\fg)\to E$ given by
	$$J'(h)(1\otimes z)\coloneqq\sum r_i(h)\otimes q(l_i(h)\otimes z),$$
	where $h\in U(\fg)$, $z\in P$, and $r_i(h),l_i(h)\in U^{[r]}(G)$ are such that $\sum r_i(h)\otimes_D l_i(h)$ is the inverse image of $1\otimes h$ under the canonical isomorphism $$can:U^{[r]}(G)\otimes_D U^{[r]}(G)\to U^{[r]}(G)\otimes U(\fg).$$
	
	By Remark 1.1(4) in \cite{Sch}, the inverse of the map $can$ sends $x\otimes\overline{y}\to \sum xS(y_{(1)})\otimes y_{(2)}$, where $\overline{y}$ is the image of $y\in U^{[r]}(G)$ under the projection $\Upsilon_{r,r}:U^{[r]}(G)\twoheadrightarrow U(\fg)$.
	
	Now fix a $U(\fg)$-comodule map $\gamma:U(\fg)\to U^{[r]}(G)$ such that $\Upsilon_{r,r}\circ\gamma=\Id_{U(\fg)}$ and $S\circ\gamma=\gamma\circ S$, where $\Upsilon_{r,r}:U^{[r]}(G)\twoheadrightarrow U(\fg)$ is as defined in Section 4 in \cite{West}. The proof of Proposition~\ref{cleft} illustrates a way to do this. We hence describe the isomorphism $J\coloneqq J'S:U(\fg)\to E$ as follows:
	$$x\mapsto \left(1\otimes_D z\mapsto \sum \gamma(x)_{(1)}\otimes_D q(S(\gamma(x)_{(2)})\otimes z)\right)$$
	for $x\in U(\fg)$ and $z\in P$.
\end{rmk}

This remark in fact shows that the two methods from \cite{West}, discussed above, are deeply related. In particular, if we compose the isomorphism $U(\fg)\xrightarrow{\sim} E$ with the $E$-action on $\Hom_{\tiny G_{r}}(P,M)$ from the second method then we recover the $U(\fg)$-action on $\Hom_{\tiny G_{r}}(P,M)$ used in the first method. In this paper we prefer to work with the second method, since the actions of $E=U(\fg)$ and $U^{[r]}(G)$ are easier to compute with in this case. This shall be most beneficial in Lemma~\ref{pChars} and in Section~\ref{s3}, where the actions of central elements in $U^{[r]}(G)$ and $U(\fg)$ are explored.

We define $\Gamma_{P}$ to be the category of irreducible left $U^{[r]}(G)$-modules which decompose as $\Di(G_{r})$-modules into a direct sum of copies of ($\Di(G_r)$-modules isomorphic to) $P$. This is a full subcategory of the category of irreducible left $U^{[r]}(G)$-modules. Furthermore, set $\modu(U(\fg))$ to be the category of finite-dimensional left $U(\fg)$-modules.

We shall examine the functor
$$\Psi_{P}:\Gamma_{P}\to\modu(E)=\modu(U(\fg))$$ which sends $M\in\Gamma_{P}$ to $\Hom_{\tiny G_{r}}(P,M)$ . 

The following theorem should also be compared with Theorem 3.1 in \cite{Wit}.

\begin{theorem}\label{equiv}
	There is an equivalence of categories between $\Gamma_P$ and $\Irr(E)$. In particular, this equivalence is obtained from the maps
	$$\Psi_{P}:\Gamma_{P}\to \Irr(E),\qquad \Psi_{P}(M)=\Hom_{\tiny G_{r}}(P,M);$$
	$$\Phi_{P}:\Irr(E)\to \Gamma_{P},\qquad \Phi_P(N)=P\otimes_{\bK} N.$$
\end{theorem}

\begin{proof}
	We maintain the convention $D=\Di(G_{r})$ to make formulas clearer.
	
	If $M\in\Gamma_{P}$, then Lemma 7.1.3 and Theorem 7.1.4 in \cite{West} show that $$\Psi_P(M)=\Hom_{\tiny D}(P,M)=\Hom_{\tiny U^{[r]}(G)}(U^{[r]}(G)\otimes_{\tiny D}P,M)$$ is a left $E$-module; that $P\otimes_{\bK}\Psi_P(M)$ is a left $U^{[r]}(G)$-module; that $P\otimes_{\bK}\Psi_P(M)$ is isomorphic to $(U^{[r]}(G)\otimes_{\tiny D}P)\otimes_{E}\Psi_P(M)$ as $U^{[r]}(G)$-modules; and that $$\eta_M:(U^{[r]}(G)\otimes_{\tiny D}P)\otimes_{E}\Psi_P(M)\to M, \qquad \eta_M(a\otimes_D z \otimes_E \phi)=\phi(a\otimes_D z)$$ is an isomorphism of $U^{[r]}(G)$-modules.
	
	Note that $\Psi_P(M)$ is an irreducible $E$-module, since if $\Psi_P(M)$ contains a proper non-trivial submodule $U$ then $$P\otimes_{\bK} U\cong (U^{[r]}(G)\otimes_{\tiny D}P)\otimes_{E} U$$ is a proper non-trivial $U^{[r]}(G)$-submodule of the irreducible $U^{[r]}(G)$-module $$M\cong (U^{[r]}(G)\otimes_{\tiny D}P)\otimes_{E} \Psi_P(M)\cong P\otimes_{\bK} \Psi_{P}(M).$$
	
	Now, suppose $N$ is an irreducible left $E$-module. It was proved in \cite[Lemma 7.1.3]{West} that 
	$$\Phi_P(N)\coloneqq P\otimes_{\bK} N\cong (U^{[r]}(G)\otimes_D P)\otimes_E N$$ is a left $U^{[r]}(G)$-module, and furthermore that the structure is such that $\Phi_P(N)$ is a direct sum of copies of $P$ as a $\Di(G_{r})$-module.
	
	We now wish to show that $\Hom_{D}(P,\Phi_P(N))\cong N$ as left $E$-modules. Define $$\sigma_N:N\to \Hom_{D}(P,\Phi_P(N)) \quad \mbox{by}\quad\sigma_N(n)(z)=z\otimes n\in  P\otimes_{\bK} N.$$ Since $$\Hom_{D}(P,\Phi_P(N))\cong \Hom_{U^{[r]}(G)}(U^{[r]}(G)\otimes_{D} P, \Phi_{P}(N))$$ as left $E$-modules and $$P\otimes_{\bK} N\cong (U^{[r]}(G)\otimes_D P)\otimes_E N$$ as left $U^{[r]}(G)$-modules, we can also write this map as $$\sigma_N:N\to \Hom_{U^{[r]}(G)}(U^{[r]}(G)\otimes_D P, (U^{[r]}(G)\otimes_D P)\otimes_E N),\qquad \sigma_N(n)(a\otimes_D z)=(a\otimes_D z)\otimes_E n$$
	for $n\in N$, $z\in P$ and $a\in U^{[r]}(G)$.
	
	It is easy to see that $\sigma_N(n)$ is a $U^{[r]}(G)$-module homomorphism from $U^{[r]}(G)\otimes_D P$ to $(U^{[r]}(G)\otimes_D P)\otimes_E N$, and also that $\sigma_N$ is a linear map. We show that $\sigma_N$ is $E$-linear. It is enough to show that for $f\in E$, $n\in N$, $z\in P$ and $a\in U^{[r]}(G)$, we have that $$(f\cdot\sigma_N(n))(a\otimes_D z)=\sigma_N(f\cdot n)(a\otimes_D z).$$ Note that $$(f\cdot\sigma_N(n))(a\otimes_D z)=\sigma_N(n)(f(a\otimes_D z))=f(a\otimes_D z)\otimes_E n,$$ while $$\sigma_N(f\cdot n)(a\otimes_D z)=(a\otimes_D z)\otimes_E (f\cdot n).$$ Since the right $E$-module structure on $U^{[r]}(G)\otimes_D P$ comes from the evaluation map, the result holds from the definition of the tensor product.
	
	Hence, $\sigma_N$ is an $E$-module homomorphism. It is clear that $\sigma_N$ is injective from the description $\sigma_N(n)(z)=z\otimes n\in P\otimes_{\bK} N$ for $n\in N$, $z\in P$. Furthermore, by above, $$\Phi_P(N)\cong\bigoplus_{i=1}^{k} P$$ as $\Di(G_{r})$-modules. Now, $k=\dim(N)$ as $\dim(\Phi_P(N))=\dim(P)\dim(N)$ and $\dim(\bigoplus_{i=1}^{k} P)=k\dim(P)$. Hence, $$\Hom_D(P,\Phi_P(N))\cong \Hom_D(P,\bigoplus_{i=1}^{k} P)=\bK^{k},$$ since $\Hom_D(P,P)=\bK$. Thus, $\dim(N)=k=\dim(\Hom_D(P,\Phi_P(N)))$. Together with the injectivity, this proves that $\sigma_N$ is an isomorphism of $E$-modules.
	
	Furthermore, $\Phi_P(N)$ is an irreducible $U^{[r]}(G)$-module since if it contains a proper non-trivial submodule $L$ then $$\Hom_{\tiny D}(P,L)\cong\Hom_{\tiny U^{[r]}(G)}(U^{[r]}(G)\otimes_D P,L)$$ is a proper non-trivial $E$-submodule of $$N\cong \Hom_{U^{[r]}(G)}(U^{[r]}(G)\otimes_{D} P, \Phi_{P}(N)) \cong \Hom_{D}(P,\Phi_P(N)),$$ contradicting the irreducibility of $N$.
	
	In conclusion, we have shown that the maps $\Psi_P$ and $\Phi_P$ are well-defined; that for any irreducible $U^{[r]}(G)$-module $M$, $\Phi_P(\Psi_P(M))\cong M$ as $U^{[r]}(G)$-modules; and that for any irreducible $E$-module $N$, $\Psi_P(\Phi_P(N))\cong N$ as $E$-modules. It is then straightforward to see that this bijection is in fact an equivalence of categories.
\end{proof}
\begin{rmk}\label{NotIrr}
	This proof in fact shows that for any $E$-module $N$, not necessarily irreducible, it is true that $N\cong \Hom_{G_{r}}(P,P\otimes_{\bK}N)=\Hom_{G_{r}}(P,(U^{[r]}(G)\otimes_D P)\otimes_E N)$ as $E$-modules.
\end{rmk}
%
%
%

For each $\bK$-algebra $R$ we consider in this section, we denote by $\underline{\Irr}(R)$ the set of isomorphism classes of irreducible $R$-modules.

\begin{cor}
	There is a bijection
	$$\Psi:\underline{\Irr}(U^{[r]}(G))\to \underline{\Irr}(\Di(G_{r}))\times \underline{\Irr}(U(\fg))$$
	which sends $M$ to $(P,\Hom_{\tiny G_{r}}(P,M))$, where $P$ is the unique (up to isomorphism) irreducible $\Di(G_{r})$-submodule of $M$. The reverse map sends $(P,N)$ to the $U^{[r]}(G)$-module $(U^{[r]}(G)\otimes_D P)\otimes_{U(\fg)} N=P\otimes_{\bK} N$.
\end{cor}

We are now in a position to give the deferred proof of Proposition 7.1.7 from \cite{West}.

\begin{prop}\label{Decomp}
	Suppose that $G$ is a reductive algebraic group over an algebraically closed field $\bK$ of positive characteristic $p$, and let $\chi\in\fg^{*}$. Let $M$ be an irreducible $U^{[r]}_\chi(G)$-module and $P$ an irreducible $\Di(G_{r})$-module such that $M\cong P\otimes\Hom_{\tiny \Di(G_{r})}(P,M)$ as $\Di(G_{r})$-modules. Then $\Hom_{\tiny \Di(G_{r})}(P,M)$ is an irreducible $U_\chi(\fg)$-module.
\end{prop}

\begin{proof}
	All that remains is to show that for $x\in\fg$, $x^p-x^{[p]}$ acts on $\Hom_D(P,M)$ as $\chi(x)^{p}$. Given $\delta\in \Di_{p^r}^{+}(G)$, we know that $\delta^{\otimes p}-\delta^p$ is central in $\Upr$. Hence, the map $$\rho(\delta^{\otimes p}-\delta^p):\Upr\otimes_{D}P\to \Upr\otimes_{D}P$$ given by left multiplication by $\delta^{\otimes p}-\delta^p$ is a $\Upr$-module endomorphism of $\Upr\otimes_{D}P$, and so lies inside $E$. However, as we know that $M$ is a $U^{[r]}_\chi(G)$-module, $\rho(\delta^{\otimes p}-\delta^p)\in E$ acts on $\Hom_{\tiny D}(P,M)$ as multiplication by $\chi(\delta)^p$.
	
	Hence, to show that $\Hom_{\tiny D}(P,M)$ is a $U_\chi(\fg)$-module, we just need that, for $\alpha\in\Phi$, $\ve_{\alpha}^p$ maps to $\rho((\ve_{\alpha}^{(p^r)})^{\otimes p})$ and, for $1\leq t\leq d$, $\vh_{t}^p-\vh_{t}$ maps to $\rho( \binom{\vh_{t}}{p^r}^{\otimes p}-{\binom{\vh_{t}}{p^r}})$ under the isomorphism $U(\fg)\cong E$. 
	
	This isomorphism was described in Remark~\ref{isom}. In particular, we know that $\ve_\alpha^p=\overline{(\ve_{\alpha}^{(p^r)})^{\otimes p}}$ and $\vh_{t}^p-\vh_{t}=\overline{\binom{\vh_{t}}{p^r}^{\otimes p}-{\binom{\vh_{t}}{p^r}}}$ for $\alpha\in \Phi$ and $1\leq t\leq d$.
	
	Observe that $$\Delta((\ve_{\alpha}^{(p^r)})^{\otimes p})=\Delta(\ve_{\alpha}^{(p^r)})^{\otimes p}=\sum_{i=0}^{p^r}(\ve_\alpha^{(i)})^{\otimes p}\otimes(\ve_\alpha^{(p^r-i)})^{\otimes p}=(\ve_{\alpha}^{(p^r)})^{\otimes p}\otimes 1 + 1\otimes (\ve_{\alpha}^{(p^r)})^{\otimes p},$$
	since $(\ve_\alpha^{(i)})^{\otimes p}=0$ for all $0<i<p^r$, while
	\begin{align*}
		\Delta(\binom{\vh_{t}}{p^r}^{\otimes p}-{\binom{\vh_{t}}{p^r}})=\Delta(\binom{\vh_{t}}{p^r})^{\otimes p} - \Delta(\binom{\vh_{t}}{p^r})\\ =
		\sum_{i=0}^{p^r}\binom{\vh_{t}}{i}^{\otimes p}\otimes \binom{\vh_{t}}{p^r-i}^{\otimes p} - \sum_{i=0}^{p^r}\binom{\vh_{t}}{i}\otimes \binom{\vh_{t}}{p^r-i}\\
		=(\binom{\vh_{t}}{p^r}^{\otimes p}-{\binom{\vh_{t}}{p^r}})\otimes 1 + 1\otimes(\binom{\vh_{t}}{p^r}^{\otimes p}-{\binom{\vh_{t}}{p^r}})
	\end{align*}
	since $\binom{\vh_{t}}{i}^{\otimes p}={\binom{\vh_{t}}{i}}$ for all $0<i<p^r$.
	
	Hence, $J'(\ve_{\alpha}^p)(1\otimes z)=1\otimes q((\ve_{\alpha}^{(p^r)})^{\otimes p}\otimes z) - (\ve_{\alpha}^{(p^r)})^{\otimes p}\otimes q(1\otimes z)$. However, the $U^{[r]}(G)$-module structure on $P$ comes through the map $U^{[r]}(G)\twoheadrightarrow \Di(G_{r+1})$, so $q((\ve_{\alpha}^{(p^r)})^{\otimes p}\otimes z)=0$. Thus, $J'(\ve_{\alpha}^p)(1\otimes z)=-(\ve_{\alpha}^{(p^r)})^{\otimes p}\otimes z$.
	
	Similarly, $J'(\vh_{t}^p-\vh_{t})(1\otimes z)=-(\binom{\vh_{t}}{p^r}^{\otimes p}-{\binom{\vh_{t}}{p^r}})\otimes z$. 
	
	By Remark 3.8 in \cite{Sch}, the algebra homomorphism $J:U(\fg)\to E$ is defined as $J=J'S$. Hence, we conclude that $J(\ve_{\alpha}^p)=\rho((\ve_{\alpha}^{(p^r)})^{\otimes p})$ for $\alpha\in\Phi$, and $J(\vh_{t}^p-\vh_{t})=\rho( \binom{\vh_{t}}{p^r}^{\otimes p}-{\binom{\vh_{t}}{p^r}})$ for $1\leq t\leq d$ (using for the latter that $\binom{\vh_{t}}{i}^{\otimes p}=\binom{\vh_{t}}{i}$ for $i<p^r$). The result follows.
\end{proof}

\begin{cor}\label{prem}
	Suppose that $G$ is a connected reductive algebraic group over an algebraically closed field $\bK$ of positive characteristic $p>0$. Suppose further that $\fg$ and $p$ are such that Premet's theorem holds (see \cite{Prem}). Let $M$ be an irreducible $U^{[r]}_\chi(G)$-module and $N$ an irreducible $\Di(G_{r})$-module such that $M\cong N\otimes\Hom_{\tiny \Di(G_{r})}(N,M)$ as $\Di(G_{r})$-modules. Then $p^{\dim(G\cdot\chi)/2}$ divides $\dim \Hom_{\tiny \Di(G_{r})}(N,M)$.
\end{cor}


%


\begin{lemma}\label{pChars}
	Let $P\in \Irr(\Di(G_{r}))$ and $N\in \Irr(U(\fg))$ with $p$-character $\chi\in\fg^{*}$ (so $N\in\Irr(U_\chi(\fg))$). Then the following results hold.
	\begin{enumerate}
		\item $$(U^{[r]}(G)\otimes_D P)\otimes_{U(\fg)} N$$
		is a left $U^{[r]}_\chi(G)$-module;
		\item $U_\chi^{[r]}(G)\otimes_D P$ is a right $U_\chi(\fg)$-module;	
		and 
		\item as $U^{[r]}_\chi(G)$-modules,
		$$(U^{[r]}(G)\otimes_D P)\otimes_{U(\fg)} N\cong (U_\chi^{[r]}(G)\otimes_D P)\otimes_{U_\chi(\fg)} N.$$
	\end{enumerate}
\end{lemma}

\begin{proof}
	(1) To show that $(U^{[r]}(G)\otimes_D P)\otimes_{U(\fg)} N$ is a left $U^{[r]}_\chi(G)$-module, it is enough to show that $\delta^{\otimes p} -\delta^p-\chi(\delta)^p$ acts on it by zero multiplication for all $\delta\in\Di^{+}_{p^{r}}(G)$. Set $\delta\in\Di^{+}_{p^r}(G)$, and let $x=\Upsilon_{r,r}(\delta)\in\fg$.
	
	Let $u\in U^{[r]}(G)$, $z\in P$ and $n\in N$. Then
	\begin{align*}
		(\delta^{\otimes p} -\delta^p-\chi(\delta)^p)\cdot (u\otimes_D z)\otimes_{\tiny U(\fg)} n=(u\otimes_D z)\cdot(x^p-x^{[p]}-\chi(x)^p)\otimes_{\tiny U(\fg)} n \\= (u\otimes_D z)\otimes_{\tiny U(\fg)} (x^p-x^{[p]}-\chi(x)^p)\cdot n=0.
	\end{align*}

	(2) To show that $U_\chi^{[r]}(G)\otimes_D P$ is a right $U_\chi(\fg)$-module, first note that $\Di(G_{r})$ is a subalgebra of $U_\chi^{[r]}(G)$, so the tensor product makes sense. We will show that $U_\chi^{[r]}(G)\otimes_D P$ is a right $E$-module, on which the left multiplication by $\delta^{\otimes p} -\delta^p-\chi(\delta)^p$ is zero for all $\delta\in\Di_{p^r}^{+}(G)$.
	
	Let $f\in\End_{\tiny U^{[r]}(G)}(U^{[r]}(G)\otimes_D P)^{op}$. We want a linear map $\widetilde{T_f}:U_\chi^{[r]}(G)\otimes_D P\to U_\chi^{[r]}(G)\otimes_D P$. By the universal property of the tensor product, it is enough to give a linear map $T_f:U_\chi^{[r]}(G)\times P\to U^{[r]}(G)\otimes_D P$ which is $\Di(G_{r})$-balanced. 
	
	Define $T_f(\overline{u},z)=\overline{f(u\otimes_D z)}$ for $u\in U^{[r]}(G)$ and $z\in P$, where $\overline{f(u\otimes_D z)}$ is the image of $f(u\otimes_D z)$ under the map $U^{[r]}(G)\otimes_D P\twoheadrightarrow U^{[r]}_\chi(G)\otimes_D P$. First, we must see that this is well-defined. Suppose $\overline{u}=\overline{v}\in U^{[r]}_\chi(G)$. Hence, $u-v\in I\unlhd U^{[r]}(G)$, where $I$ is the ideal generated by $\delta^{\otimes p} -\delta^p-\chi(\delta)^p$ for $\delta\in\Di_{p^r}^{+}(G)$. So $f((u-v)\otimes_D z)\in I\otimes_D P$, so $\overline{f((u-v)\otimes_D z)}=0$. Furthermore, for $d\in\Di(G_{r})$, $$T_f(\overline{u}\cdot d,z)=T_f(\overline{ud},z)=\overline{f(ud\otimes_D z)}=\overline{f(u\otimes_D dz)}=T_f(\overline{u},d\cdot z).$$
	
	Hence, we obtain a linear map $\widetilde{T_f}:U_\chi^{[r]}(G)\otimes_D P\to U_\chi^{[r]}(G)\otimes_D P$. It is straightforward to see that $\widetilde{T_f}\widetilde{T_g}=\widetilde{T_{fg}}$, so $U_\chi^{[r]}(G)\otimes_D P$ is a right $E$-module. One may then check that the action of left multiplication by $\delta^{\otimes p} -\delta^p-\chi(\delta)^p$ is zero for all $\delta\in\Di_{p^r}^{+}(G).$ 
	
	Hence $U_\chi^{[r]}(G)\otimes_D P$ is a right $U_\chi(\fg)$-module. 
	
	(3) All that remains is to show the isomorphism $(U^{[r]}(G)\otimes_D P)\otimes_{U(\fg)} N\cong (U_\chi^{[r]}(G)\otimes_D P)\otimes_{U_\chi(\fg)} N$.
	
	Define the map $F:(U^{[r]}(G)\otimes_D P)\times N\to (U_\chi^{[r]}(G)\otimes_D P)\otimes_{U_\chi(\fg)} N$ by sending the elements $(u\otimes_D z,n)$ to $(\overline{u}\otimes_D z)\otimes_{\tiny U_\chi(\fg)} n$, where $\overline{u}=u+I$. It is easy to see that is map is a well-defined $U_\chi^{[r]}(G)$-module homomorphism. It is also $U(\fg)$-balanced:
	$$F((u\otimes_D z)\cdot f,n)= \overline{f(u\otimes_D z)}\otimes_{\tiny U_\chi(\fg)} n= (u\otimes_D z)\otimes_{\tiny U_\chi(\fg)} \overline{f}\cdot n = F(u\otimes_D z,f\cdot n),$$
	where $u\in U^{[r]}(G)$, $z\in P$, $n\in N$, $f\in E\cong U(\fg)$ and $\overline{f}=f+J\in E/J$, where $J$ is the ideal in $E$ generated by left multiplications by the elements $\delta^{\otimes p} -\delta^p-\chi(\delta)^p$ for $\delta\in\Di_{p^{r}}^{+}(G)$. Hence, there is a $U^{[r]}_\chi(G)$-module homomorphism $\widetilde{F}:(U^{[r]}(G)\otimes_D P)\otimes_{U(\fg)} N\to (U_\chi^{[r]}(G)\otimes_D P)\otimes_{U_\chi(\fg)} N$.
	
	Furthermore, we define $H:(U_\chi^{[r]}(G)\otimes_D P)\times N\to(U^{[r]}(G)\otimes_D P)\otimes_{U(\fg)} N$ by sending the elements $(\overline{u}\otimes_D z,n)$ to $(u\otimes_D z)\otimes_{\tiny U(\fg)} n$. This map is well-defined, since $(U^{[r]}(G)\otimes_D P)\otimes_{U(\fg)} N$
	is a $U^{[r]}_\chi(G)$-module, and a homomorphism of $U^{[r]}_\chi(G)$-modules. It is also $U_\chi(\fg)$-balanced:
	$$H((\overline{u}\otimes_D z)\cdot \overline{f},n)= f(u\otimes_D z)\otimes_{\tiny U_\chi(\fg)} n= (u\otimes_D z)\otimes_{\tiny U_\chi(\fg)} f\cdot n = F((u\otimes_D z),\overline{f}\cdot n),$$
	where $u\in U^{[r]}(G)$, $z\in P$, $n\in N$, $f\in E\cong U(\fg)$ and $\overline{f}=f+J\in E/J$. This gives a $U^{[r]}_\chi(G)$-module homomorphism $\widetilde{H}:(U_\chi^{[r]}(G)\otimes_D P)\otimes_{U_\chi(\fg)} N \to (U^{[r]}(G)\otimes_D P)\otimes_{U(\fg)} N$.
	
	It is straightforward to see that $\widetilde{F}$ and $\widetilde{H}$ are inverse to each other. The result follows.
\end{proof}

\begin{cor}\label{chibij}
	There is a bijection
	$$\Psi_\chi:\underline{\Irr}(U^{[r]}_\chi(G))\to \underline{\Irr}(\Di(G_{r}))\times \underline{\Irr}(U_\chi(\fg))$$
	which sends $M$ to $(P,\Hom_{\tiny G_{r}}(P,M))$, where $P$ is the unique (up to isomorphism) irreducible $\Di(G_{r})$-submodule of $M$. The inverse map sends $(P,N)$ to $(U^{[r]}_\chi(G)\otimes_{\tiny \Di(G_{r})} P)\otimes_{U_\chi(\fg)} N\cong P\otimes_{\bK} N$.
\end{cor}

\subsection{Teenage Verma modules}
\label{s2.2}

We can use the previous section to deduce some structural results about irreducible $U^{[r]}_\chi(G)$-modules. We start by defining the following vector subspace of $U^{[r]}(G)$, using the $\llbracket\cdot\rrbracket$ notation from \cite{West}:
$$\widehat{U^{[r]}(B)}\coloneqq\bK-\mbox{span}\{\prod_{\alpha\in\Phi^+}\ve_{\alpha}^{\llbracket i_\alpha\rrbracket}\prod_{t=1}^d\binom{\vh_{t}}{\llbracket k_t\rrbracket}\prod_{\alpha\in\Phi^+}\ve_{-\alpha}^{\llbracket j_\alpha\rrbracket }\quad :\quad 0\leq i_\alpha,k_t,\,0\leq j_\alpha<p^{r}\,\}.$$

This vector space is in fact a subalgebra of $U^{[r]}(G)$ by the multiplication equations given in \cite{Cham}. Furthermore, the Hopf algebra structure on $U^{[r]}(G)$ makes $\widehat{U^{[r]}(B)}$ into a Hopf subalgebra of $U^{[r]}(G)$.

Clearly $\Di(G_{r})$ is a subalgebra of $\widehat{U^{[r]}(B)}$, it is normal since it is normal in $U^{[r]}(G)$, and $\widehat{U^{[r]}(B)}$ is free as both a left and right $\Di(G_{r})$-module.

From \cite{West}, we know that the map $\Upsilon_{r,r}:U^{[r]}(G)\to U(\fg)$ is a surjective Hopf algebra homomorphism. It is easy to see from the bases that this map restricts to a surjective Hopf algebra homomorphism $\widehat{U^{[r]}(B)}\twoheadrightarrow U(\fb)$, with kernel $\widehat{U^{[r]}(B)}\Di^{+}(G_{r})=\Di^{+}(G_{r})\widehat{U^{[r]}(B)}$. In particular, $\Di(G_{r})\subset \widehat{U^{[r]}(B)}$ is a $U(\fb)$-module extension, with $\Di(G_{r})=\widehat{U^{[r]}(B)}^{co U(\fb)}$.

\begin{lemma}
	Let $P\in\Irr(\Di(G_{r})$. Then $\End_{\widehat{U^{[r]}(B)}}(\widehat{U^{[r]}(B)}\otimes_D P)\cong U(\fb)$.  
\end{lemma}

\begin{proof}
	This follows as in Lemma 7.1.5 from \cite{West}, since $\widehat{U^{[r]}(B)}$ is a subalgebra of $U^{[r]}(G)$.
\end{proof}

It is straightforward to see that the proof of Theorem 7.1.4 in \cite{West} and the proof of Theorem~\ref{equiv} above hold similarly in this context. In other words, we have the following proposition.

\begin{prop}\label{bij}
	There is a bijection
	$$\widehat{\Psi}:\underline{\Irr}(\widehat{U^{[r]}(B)})\to \underline{\Irr}(\Di(G_{r}))\times \underline{\Irr}(U(\fb))$$
	which sends $M$ to $(P,\Hom_{\tiny G_{r}}(P,M))$, where $P$ is the unique (up to isomorphism) irreducible $\Di(G_{r})$-submodule of $M$. The inverse map sends $(P,N)$ to the $\widehat{U^{[r]}(B)}$-module $(\widehat{U^{[r]}(B)}\otimes_D P)\otimes_{U(\fb)} N=P\otimes_{\bK} N$.
\end{prop}

Applying Lemma~\ref{Decomp} and Lemma~\ref{pChars} in this context, we get the following corollary.

\begin{cor}
	The bijection in Proposition~\ref{bij} restricts to a bijection
	$$\widehat{\Psi_\chi}:\underline{\Irr}(\widehat{U^{[r]}_\chi(B)})\to \underline{\Irr}(\Di(G_{r}))\times \underline{\Irr}(U_\chi(\fb)).$$
\end{cor}

Assume from now on that $\chi(\fn^{+})=0$. It is well known (see, for example, \cite{Jan}) that, if $N\in\Irr(U_\chi(\fb))$, then $N=\bK_\lambda$ for some $\lambda\in\Lambda_\chi$, where $\bK_\lambda$ denotes the 1-dimensional $\fb$-module on which $\fn^{+}$ acts trivially and $h\in\fh$ acts through multiplication by $\lambda(h)$. Recall here that
$$\Lambda_\chi\coloneqq\{\lambda\in \fh^{*}\,\vert\,\lambda(h)^p-\lambda(h)=\chi(h)^p\;\mbox{for all}\; h\in\fh\}.$$

Hence, there is a bijection, 
$$\widehat{\Psi}:\underline{\Irr}(\widehat{U^{[r]}_\chi(B)})\to \underline{\Irr}(\Di(G_{r}))\times \Lambda_\chi.$$

In other words, every irreducible $\Di(G_{r})$-module $P$ can be extended to an irreducible $\widehat{U^{[r]}_\chi(B)}$-module, and there is one such way to do this for each $\lambda\in \Lambda_\chi$. For each $\lambda\in\Lambda_\chi$, we can hence define the $U^{[r]}_\chi(G)$-module
\begin{equation*}
	\begin{split}
		U_\chi^{[r]}(G)\otimes_{\tiny \widehat{U^{[r]}_\chi(B)}} (P\otimes_\bK \bK_\lambda) & = U_\chi^{[r]}(G)\otimes_{\tiny \widehat{U^{[r]}_\chi(B)}}(\widehat{U^{[r]}_\chi(B)}\otimes_D P)\otimes_{\tiny U_\chi(\fb)}\bK_\lambda \\
		& \stackrel{\star}{=} (U_\chi^{[r]}(G)\otimes_{\tiny \widehat{U^{[r]}_\chi(B)}}\widehat{U^{[r]}_\chi(B)}\otimes_D P)\otimes_{\tiny U_\chi(\fb)} \bK_\lambda\\
		& =(U^{[r]}_\chi(G)\otimes_D P)\otimes_{\tiny U_\chi(\fb)} \bK_\lambda\\
		& =(U^{[r]}_\chi(G)\otimes_D P)\otimes_{\tiny U_\chi(\fg)} U_\chi(\fg)\otimes_{\tiny U_{\chi}(\fb)} \bK_\lambda\\
		& =(U^{[r]}_\chi(G)\otimes_D P)\otimes_{\tiny U_\chi(\fg)} Z_\chi(\lambda)\\
		& =P\otimes_{\bK} Z_\chi(\lambda).
	\end{split}
\end{equation*}

Here, equality ($\star$) follows from an easy check. 

We call this $U^{[r]}_\chi(G)$-module the {\bf teenage Verma module} $Z^r_\chi(P,\lambda)$. Note that $\dim(Z^r_\chi(P,\lambda))=p^{\dim(\fn^{-})}\dim(P)$. Frobenius reciprocity them gives the following proposition, proving both conjectures from Subsection 6.5 in \cite{West}.

\begin{prop}\label{tVm}
	Every irreducible $U^{[r]}_\chi(G)$-module is a quotient of a teenage Verma module $Z_\chi^r(P,\lambda)$ for some $P\in\Irr(\Di(G_{r}))$ and $\lambda\in\Lambda_\chi$.
\end{prop}

Despite the fact that baby Verma modules and teenage Verma modules need not be irreducible, the following lemma shows that the correspondence in Corollary~\ref{chibij} can be extended to these modules.

\begin{lemma}\label{BVCorresp}
	For $P\in\Irr(\Di(G_{r}))$ and $\lambda\in \Lambda_\chi$, $\Hom_{G_{r}}(P,Z_\chi^r(P,\lambda))\cong Z_\chi(\lambda)$ as left $U_\chi(\fg)$-modules.
\end{lemma}

\begin{proof}
	This follows directly from Remark~\ref{NotIrr}.
\end{proof}

We also obtain the following structural result.	

%
%

\begin{prop}\label{quotCorresp}
	Suppose $M\in\Irr(U_\chi^{[r]}(G))$, $P\in\Irr(\Di(G_{r}))$ and $N\in\Irr(U_\chi(\fg))$ such that $\Psi_{\chi}(M)=(P,N)$. Then $M$ is an irreducible quotient of $Z^r_\chi(P,\lambda)$ if and only if $N$ is an irreducible quotient of $Z_\chi(\lambda)$.
\end{prop}

\begin{proof}
	($\implies$) By definition of $\Psi_\chi$ and Lemma~\ref{BVCorresp}, $N=\Hom_{G_{r}}(P,M)$ and $Z_\chi(\lambda)=\Hom_{G_{r}}(P,Z^r_\chi(P,\lambda))$. Let $\pi:Z_\chi^r(P,\lambda)\to M$ be the given surjection. We then define the map $\eta:Z_\chi(\lambda)\to N$ by defining the map $\eta:\Hom_{G_{r}}(P,Z^r_\chi(P,\lambda))\to \Hom_{G_{r}}(P,M)$ as $\eta(f)(z)=\pi f(z)$ for $f\in \Hom_{G_{r}}(P,Z^r_\chi(P,\lambda))$ and $z\in P$. It is straightforward to check that this is an $E$-module homomorphism, hence a $U(\fg)$-module homomorphism, hence a $U_\chi(\fg)$-module homomorphism. It is surjective as $N$ is irreducible.
	
	($\impliedby$) By the definitions of $\Psi_{\chi}$ and $Z_\chi^r(P,\lambda)$, $M=(U_\chi^{[r]}(G)\otimes_{D} P)\otimes_{\tiny U_\chi(\fg)} N$ and $Z_\chi^r(P,\lambda)=(U_\chi^{[r]}(G)\otimes_{D} P)\otimes_{\tiny U_\chi(\fg)} Z_\chi(\lambda)$. The result then follows from the functoriality of the tensor product and the irreducibility of $M$.
\end{proof}

\subsection{Consequences}
\label{s2.3}

From now on, let us make the following assumptions (see Chapter 6 in \cite{Jan2} for more details):

(H1) The derived group of $G$ is simply-connected;

(H2) The prime $p$ is good for $G$; and

(H3) There is a non-degenerate $G$-invariant bilinear form on $\fg$.

In particular, (H3) gives rise to an isomorphism of $G$-modules $\fg\to\fg^{*}$. This allows us to transfer properties of elements of $\fg$ to properties of elements of $\fg^{*}$. For example, we say that $\chi\in\fg^{*}$ is {\bf semisimple} if the corresponding element $x\in\fg$ is semisimple (in fact this is equivalent to the requirement that $g\cdot\chi(\fn^{+}\oplus\fn^{-})=0$ for some $g\in G$, under the coadjoint action). Similarly, we say that $\chi\in\fg^{*}$ is {\bf nilpotent} if the corresponding element $x\in\fg$ is nilpotent (this is equivalent to the requirement that $g\cdot\chi(\fb)=0$ for some $g\in G$, under the coadjoint action).

Furthermore, we say that $x\in\fg$ is {\bf regular} if $\dim(C_G(x))=\dim(\fh)$, where $C_G(x)\coloneqq\{g\in G\,\vert g\cdot x=x\}$. We hence say that $\chi\in\fg^{*}$ is {\bf regular} if the corresponding $x\in\fg$ is regular - this is equivalent to the requirement that $\dim(C_G(\chi))=\dim(\fh)$, where $C_G(\chi)\coloneqq\{g\in G\,\vert g\cdot \chi=\chi\}$.

With these definitions in mind, we get the following proposition.

\begin{theorem}\label{regular}
	Let $M$ be an irreducible $U^{[r]}_\chi(G)$-module, for $\chi\in\fg^{*}$, and let $P$ be the unique (up to isomorphism) irreducible $\Di(G_{r})$-submodule of $M$.  The following results hold.
	\begin{enumerate}
		\item There exists $\lambda\in\Lambda_\chi$ such that $M$ is an irreducible quotient of $Z_\chi^r(P,\lambda)$.
		\item If $\chi$ is regular, then there exists  $P\in\Irr(\Di(G_{r}))$ and $\lambda\in\Lambda_\chi$ such that $M\cong Z_\chi^r(P,\lambda)$.
		\item If $\chi$ is regular semisimple then $Z_\chi^r(P,\lambda)\cong Z_\chi^r(\widetilde{P},\mu)$ if and only if $P\cong\widetilde{P}$ and $\lambda=\mu$.
		\item If $\chi$ is regular nilpotent and $\chi(\ve_{-\alpha})\neq 0$ for all $\alpha\in\Pi$, then $Z_\chi^r(P,\lambda)\cong Z_\chi^r(\widetilde{P},\mu)$ if and only if $P\cong\widetilde{P}$ and $\lambda\in W_\bullet \mu$, where $W$ is the Weyl group of $\Phi$ and $\bullet$ represents the dot-action.
	\end{enumerate}
\end{theorem}

\begin{proof}
	(1) By above, there exists $Q\in\Irr(\Di(G_{r}))$ and $\lambda\in\Lambda_\chi$ such that $M$ is an irreducible quotient of $Z_\chi^r(Q,\lambda)$. Frobenius reciprocity then shows that 
	$$\Hom_{\tiny U^{[r]}_\chi(G)}(Z_\chi^r(Q,\lambda),M)\cong\Hom_{\tiny \widehat{U^{[r]}_\chi(B)}}(Q\otimes_{\bK}\bK_\lambda,M).$$
	
	In particular, as $M\neq 0$, the $\Di(G_{r})$-module $Q\subset Z_\chi^r(Q,\lambda)$ is not in the kernel of the surjection $\pi:Z_\chi^r(Q,\lambda)\twoheadrightarrow M$. Hence, the surjection restricts to a $\Di(G_{r})$-isomorphism $Q\to\pi(Q)$, so $Q$ is an irreducible $\Di(G_{r})$-submodule of $M$. As a result, $Q\cong P$, and we can say that $M$ is an irreducible quotient of $Z_\chi^r(P,\lambda)$ for some $\lambda\in\Lambda_\chi$.
	
	(2) The bijection $\Psi_{\chi}$ sends $M$ to the pair $(P,N)$ for some $N\in\Irr(U_\chi(\fg))$, and $\dim(M)=\dim(P)\dim(N)$. Since $\chi$ is regular, $\dim(N)=p^{\dim(\fn^{-})}$. 
	
	However, by (1), $M$ is an irreducible quotient of $Z_\chi^r(P,\lambda)$ for some $\lambda\in\Lambda_\chi$. Furthermore, $\dim(Z_\chi^r(P,\lambda))=p^{\dim(\fn^{-})}\dim(P)$. Hence, $M\cong Z_\chi^r(P,\lambda)$.
	
	(3) Suppose $Z_\chi^r(P,\lambda)\cong Z_\chi^r(\widetilde{P},\mu)$. The $U^{[r]}_\chi(G)$-module $Z_\chi^r(P,\lambda)$ is an irreducible module containing $P$, while $Z_\chi^r(\widetilde{P},\mu)$ is an irreducible $U^{[r]}_\chi(G)$-module containing $\widetilde{P}$. Since each irreducible $U^{[r]}_\chi(G)$-module contains a unique irreducible $\Di(G_{r})$-submodule, we obtain that $P$ and $\widetilde{P}$ are isomorphic $\Di(G_{r})$-modules.
	
	Hence, $$\Hom_{G_{r}}(P,Z_\chi^r(P,\lambda))\cong\Hom_{G_{r}}(\widetilde{P},Z_\chi^r(\widetilde{P},\mu)),$$
	and so $$Z_\chi(\lambda)\cong Z_\chi(\mu).$$
	
	By \cite[B.10]{Jan}, $\lambda=\mu$.
	
	(4) As in (3), if $Z_\chi^r(P,\lambda)\cong Z_\chi^r(\widetilde{P},\mu)$ then $Z_\chi(\lambda)\cong Z_\chi(\mu)$. Hence, by \cite[Proposition 10.5]{Jan2}, $\lambda\in W_\bullet\mu +pX$.
	
\end{proof}

Since all irreducible $U^{[r]}(G)$-modules have finite dimension, we can determine the maximal dimension of an irreducible $U^{[r]}(G)$-module, $\sup\{\dim(M)\,\vert\,M\in\Irr(U^{[r]}(G))\}$.

\begin{cor}\label{maxdim}
	The maximal dimension of an irreducible $U^{[r]}(G)$-module is $p^{(r+1)\dim(\fn^{-})}$, and it is attained.
\end{cor}

\begin{proof}
	Since every irreducible $U^{[r]}(G)$-module is an irreducible quotient of $Z^r_\chi(P,\lambda)$ for some $\chi\in\fg^{*}$, $\lambda\in\Lambda_\chi$ and irreducible $\Di(G_r)$-module $P$, and since the dimension of $Z^r_\chi(P,\lambda)$ depends only on $P$, the maximal dimension of an irreducible $U^{[r]}(G)$-module is at most $$\max_{\tiny P\in\Irr(\Di(G_{r}))}\{\dim(Z^r_\chi(P,\lambda))\}=\max_{\tiny P\in\Irr(\Di(G_{r}))}\{(p^{\dim(\fn^{-})}\dim(P))\}.$$ The maximal dimension of an irreducible $\Di(G_{r})$-module is $p^{r\dim(\fn^{-})}$, coming from the Steinberg weight $St$. In particular, if we choose $P=L_r(St)$ and $\chi$ regular, then $Z^r_\chi(P,\lambda)$ is an irreducible $U^{[r]}(G)$-module of dimension $p^{(r+1)\dim(\fn^{-})}$, and the result follows.
\end{proof}

Recall further that, given $x\in\fg$, there exist $x_s,x_n\in\fg$ such that $x=x_s+x_n$, $x_s$ is semisimple in $\fg$, $x_n$ is nilpotent in $\fg$ and $[x_s,x_n]=0$. We call $x=x_s+x_n$ a {\bf Jordan decomposition} of $x$. If, under the $G$-module isomorphism $\fg\to\fg^{*}$, $x$ maps to $\chi$, $x_s$ maps to $\chi_s$ and $x_n$ maps to $\chi_n$, we call $\chi=\chi_s+\chi_n$ a Jordan decomposition of $\chi$.

Given $\chi\in\fg^{*}$, we define $\fc_{\fg}(\chi)\coloneqq\{y\in \fg\,\vert\,\chi([\fg,y])=0\}$. Under our assumptions, $C_G(\chi_s)$ is a Levi subgroup of $G$ with Lie algebra $\fc_{\fg}(\chi_s)$ (see \cite[Lemma 3.2]{BG}). Hence, there exists a parabolic subgroup $P_{\chi_s}$ of $G$ which is a semi-direct product of $C_G(\chi_s)$ with its unipotent radical $U_{P_{\chi_s}}$. Letting $\fu=\Lie(U_{P_{\chi_s}})$ and $\fp=\Lie(P_{\chi_s})$, we get that $\fp=\fc_\fg(\chi_s)\oplus\fu$. Work of Friedlander and Parshall in \cite{FP} shows that there is a equivalence of categories
$$\modu(U_\chi(\fg))\longleftrightarrow\modu(U_\chi(\fc_{\fg}(\chi_s)))$$
which sends $N\in\modu(U_\chi(\fg))$ to the fixed point set $N^{\fu}\in\modu(U_\chi(\fc_{\fg}(\chi_s)))$, and sends $V\in \modu(U_\chi(\fc_{\fg}(\chi_s)))$ to $U_\chi(\fg)\otimes_{U_\chi(\fp)}V\in \modu(U_\chi(\fg))$, where $\fu$ acts on $V$ as $0$.

Furthermore, letting $\mu=\chi\vert_{\fc_{\fg}(\chi_s)}$, there is another equivalence of categories
$$\modu(U_\mu(\fc_{\fg}(\chi_s))\longleftrightarrow \modu(U_{\mu_n}(\fc_{\fg}(\chi_s)))$$
which sends $V\in \modu(U_\mu(\fc_{\fg}(\chi_s)))$ to $V\otimes W\in \modu(U_{\mu_n}(\fc_{\fg}(\chi_s)))$ and $V\in \modu(U_{\mu_n}(\fc_{\fg}(\chi_s)))$ to $V\otimes W^{*}\in \modu(U_\mu(\fc_{\fg}(\chi_s)))$, where $W$ is a irreducible $U_{\mu_s}(\fc_{\fg}(\chi_s)/[\fc_{\fg}(\chi_s),\fc_{\fg}(\chi_s)])$-module (necessarily 1-dimensional) viewed as a $\fg$-module.

Both of these equivalences of categories send baby Verma modules to baby Verma modules.

\begin{cor}
	Keep the notation from the preceding paragraph. There is a bijection
	$$\Psi_\chi:\underline{\Irr}(U^{[r]}_\chi(G))\to \underline{\Irr}(\Di(G_{r}))\times \underline{\Irr}(U_{\mu_n}(\fc_{\fg}(\chi_s)))$$
	which sends $M$ to $(P,\Hom_{\tiny G_{r}}(P,M)^{\fu}\otimes W^{*})$, where $P$ is the unique (up to isomorphism) irreducible $\Di(G_{r})$-submodule of $M$. The inverse map sends $(P,V)$ to $(U^{[r]}_\chi(G)\otimes_{\tiny \Di(G_{r})} P)\otimes_{U_\chi(\fp)} (V\otimes W)\cong P\otimes_{\bK} (U_\chi(\fg)\otimes_{U_\chi(\fp)}(V\otimes W))$.
\end{cor}

In particular, this result means that to study the irreducible $U^{[r]}_\chi(G)$-modules, one may always assume that $\chi\vert_{\fc_{\fg}(\chi_s)}$ is nilpotent, and hence that $\chi$ vanishes on $\fb\cap\fc_{\fg}(\chi_s)$.

Recall that we say that $\chi\in\fg^{*}$ has {\bf standard Levi form} if $\chi(\fb)=0$ and there exists a subset $I\subseteq\Pi$ with 
$\chi(\ve_{-\alpha})= 0$ if and only if $\alpha\in \Phi^{+}\setminus I$.

\begin{dfn}
	We say that $\chi\in\fg^{*}$ has {\bf almost standard Levi form} if $(\chi\vert_{\fc_{\fg}(\chi_s)})_n$ has standard Levi form.
\end{dfn}

\begin{prop}
	Suppose that $\chi\in\fg^{*}$ has almost standard Levi form. Let $P\in\Irr(\Di(G_{r}))$ and $\lambda\in\Lambda_\chi$. Then the $U^{[r]}_\chi(G)$-module $Z_\chi^r(P,\lambda)$ has a unique irreducible quotient. 
\end{prop}

\begin{proof}
	Since $\mu_n\coloneqq(\chi\vert_{\fc_{\fg}(\chi_s)})_n$ has standard Levi form, each $Z_{\mu_n}(\tau)$ for $\tau\in\Lambda_{\mu_n}$ has a unique irreducible quotient. Since there is an equivalence of categories between $\modu(U_{\mu_n}(\fc_{\fg}(\chi_s)))$ and $\modu(U_\chi(\fg))$ which sends baby Verma modules to baby Verma modules, it follows that each $Z_\chi(\lambda)$ has a unique irreducible quotient. The result then follows from Proposition~\ref{quotCorresp}.
\end{proof}

If $\chi\in\fg^{*}$ has almost standard Levi form, we shall write $L_\chi^r(P,\lambda)$ for the unique irreducible quotient of $Z_\chi^r(P,\lambda)$. Proposition 10.8 in \cite{Jan2} gives the following isomorphism condition on these modules, where $W_I$ is the subgroup of the Weyl group generated by simple reflections corresponding to simple roots in $I$.

\begin{cor}
	Suppose that $\chi\in\fg^{*}$ has almost standard Levi form corresponding to the subset $I$ of the simple roots of $\fc_{\fg}(\chi_s)$. Let $P,Q\in\Irr(\Di(G_{r}))$ and $\lambda,\widetilde{\lambda}\in\Lambda_\chi$. Then $L_\chi^r(P,\lambda)\cong L_\chi^r(Q,\widetilde{\lambda})$ if and only if $P\cong Q$ and $\widetilde{\lambda}\in W_{I\bullet}\lambda$.
\end{cor}
%
%
%
%
%
%
%
%
%

\section{The Azumaya Locus of $U^{[r]}(G)$}
\label{s3}
\subsection{Azumaya and pseudo-Azumaya loci}
\label{s3.1}
Let $R$ be a $\bK$-algebra, where $\bK$ is an algebraically closed field (of arbitrary characteristic), which is module-finite over its centre $Z=Z(R)$. Suppose further that $Z$ is an affine $\bK$-algebra (i.e. $Z$ is finitely generated as a $\bK$-algebra). One can observe that these conditions guarantee the existence of a bound on the dimensions of irreducible $R$-modules.

These conditions further imply that $R$ is a PI ring, i.e. that there exists a (multilinear) $\bZ$-polynomial $f$ such that $f(r_1,\ldots,r_k)=0$ for all $r_1,\ldots,r_k\in R$. For $n\in\bN$, we define the polynomial $g_n$ as in Chapter 1.4 of \cite{Row2} (see Proposition 1.4.10 in particular). This is an $n^2$-normal polynomial ({\bf $n^2$-normal} meaning $g_n$ is linear and alternating in its first $n^2$ variables). We then say that $R$ has {\bf PI-degree} $m$ if $R$ satisfies all multilinear identities of $M_m(\bZ)$ (that is to say, all multilinear $\bZ$-polynomials which vanish on $M_m(\bZ)$) and $g_m(R)\coloneqq\{g_m(r_1,\ldots,r_k)\,\vert\,r_1,\ldots,r_k\in R\}$ is not the zero set. If $R$ has PI-degree $m$, then $g_m(r_1,\ldots,r_k)\in Z$ for all $r_1,\ldots,r_k\in R$.

We define, as in \cite{Row2}, the following sets:
$$\Spec_m(R)\coloneqq\{P\in\Spec(R)\,\vert\, g_m(R)\not\subseteq P\},\qquad \Spec_m(Z)\coloneqq\{Q\in\Spec(Z)\,\vert\, g_m(R)\not\subseteq Q\},$$
where $\Spec(R)$ is defined to be the set of prime ideals in $R$. One can check that, if $R$ has PI-degree $m$ and $P$ is a prime ideal of $R$, $\PIDeg(R)\geq\PIDeg(R/P)$ and this inequality is an equality precisely when $P\in\Spec_m(R)$.

Given a central subalgebra $C$ of $R$, we say, as in Definition 5.3.23 in \cite{Row}, that $R$ is {\bf Azumaya} over $C$ if 

(i) $R$ is a faithful and finitely generated projective $C$-module; and

(ii) the canonical map $R\otimes_C R^{op}\to \End_C(R)$, which sends $a\otimes b$ to the map $x\mapsto axb$, is a $\bK$-algebra isomorphism.

If $C=Z$, we will simply call $R$ an Azumaya algebra. We furthermore say that $R$ is Azumaya over $C$ {\bf of constant rank $t$} if $R_I$ is a free module of rank $t$ over $C_I$ for all prime ideals $I$ of $C$ \cite[Definition 2.12.21]{Row}. By Remark 1.8.36 in \cite{Row}, we observe that if $R$ is Azumaya over $C$ of constant rank $t$ then, for each prime ideal $I$ of $C$, $R_I$ is also Azumaya over $C_I$ of constant rank $t$. 


Given a prime ideal $Q$ in $Z$, we define $R_Q$ to be the localization of $R$ at the multiplicatively closed central subset $Z\setminus Q$. In other words, $R_Q\coloneqq\{rs^{-1}\,\vert\,r\in R, s\in Z\setminus Q\}$, where $r_1s_1^{-1}=r_2s_2^{-1}$ if and only if there exists $s\in Z\setminus Q$ such that $s(r_1s_2-r_2s_1)=0$. We denote by $Z_Q$ the usual localization of $R\setminus Q$ in $Z$. By \cite{Row2}, $Z_Q\subseteq Z(R_Q)$ with equality if $Z\setminus Q$ is regular in $R$ (i.e. for any $s\in Z\setminus Q$, $r\in R$, $sr=0$ implies $r=0$).

Note that Theorem 5.3.24 in \cite{Row} implies that if $R_Q$ is Azumaya over $Z_Q$ then $Z_Q=Z(R_Q)$. The following lemma follows from Section 5.3 in \cite{Row}.

\begin{lemma}
	$R_Q$ is Azumaya over $Z_Q$ if and only if $Z_Q=Z(R_Q)$ and $R_Q$ is Azumaya over its centre. Either of these conditions is satisfied if, for example, $Z\setminus Q$ is regular in $R$ and $R_Q$ is Azumaya over its centre.
\end{lemma}

The {\bf Azumaya locus} $\mathcal{A}_R$ of $R$ is hence defined to be the set of maximal ideals $\fm$ in $Z$ such that $R_\fm$ is an Azumaya algebra over $Z_\fm$. If $R$ is prime, this is precisely the definition of Azumaya locus given in \cite{BGo}.


We shall further define the {\bf pseudo-Azumaya locus} of $R$, $\mathcal{PA}_R$, as 
$$\mathcal{PA}_R\coloneqq\{\ann_Z(M)\,\vert\,M\,\mbox{an irreducible left}\, R\mbox{-module of maximal dimension}\}.$$
The next theorems shall show how the Azumaya and pseudo-Azumaya loci are connected.


\begin{theorem}\label{PIDeg}
	Let $R$ be a $\bK$-algebra, where $\bK$ is an algebraically closed field, which is module-finite over its centre $Z=Z(R)$, and assume that $Z$ is affine. Let $J(R)$ be the Jacobson radical of $R$. Then the following results hold.
	\begin{enumerate}
		\item The ring $R/J(R)$ has PI-degree $d$, where $d$ is the maximal dimension of an irreducible (left) $R$-module. 
		\item If $R$ has PI-degree $m$, then $m=d$ if and only if there exists a primitive ideal $A$ in $\Spec_m(R)$.
	\end{enumerate}
\end{theorem}

\begin{proof}
	
	(1) Observe that for an irreducible $R$-module $M$ with annihilator $A=\ann_R(M)$, $R/A$ is a finite dimensional, simple algebra over $Z/\fm$, where $\fm=A\cap Z$. This holds because $M$ is a faithful $R/A$-module, so $R/A$ embeds in $\End_{\bK}(M)$. In particular, $R/A\cong M_{n_A}(\bK)$ by the algebraically closed nature of the field $\bK$, for some $n_A\in\bN$. Hence, every irreducible $R/A$-module has dimension $n_A$. In particular, 
	$$d=\max_{\tiny A\lhd R\,\, \mbox{primitive}}\{n_A\}.$$
	
	
	Furthermore, Kaplansky's Theorem tells us that, for a primitive ideal $A$ of $R$, the PI-degree of $R/A$ is also $n_A$.	Hence, for any primitive ideal $A$,
	$$\PIDeg(R/A)=n_A\leq d.$$
	In particular, this says that if $f$ is a multilinear identity of $M_d(\bZ)$ then $f(R)$ is a subset of all primitive ideals of $R$. Thus $R/J(R)$ satisfies all the multilinear identities of $M_d(\bZ)$.
	
	Also, if $M$ is an irreducible $R$-module of maximal dimension then $\PIDeg(R/\ann_R(M))=d$. Hence $g_d(R)\not\subseteq\ann_R(M)$, and thus $g_d(R)\not\subseteq J(R)$. So $g_d(R/J(R))\neq 0$.
	
	This precisely says that $R/J(R)$ has PI-degree $d$.
	
	(2) We know that $\PIDeg(R/\ann_R(M))=d$ when $M$ is an irreducible left $R$-module of maximal dimension. Thus, when $m=d$,  $\PIDeg(R)=\PIDeg(R/\ann_R(M))$ and so $\ann_R(M)\in\Spec_m(R)$.
	
	On the other hand, if there exists a primitive ideal $A\in\Spec_m(R)$ then 
	$$m=\PIDeg(R)=\PIDeg(R/A)\leq\PIDeg(R/J(R))\leq\PIDeg(R)$$
	and the result follows.
	%
\end{proof}

If $R$ has PI-degree $d$, the maximal dimension of an irreducible (left) $R$-module, then the pseudo-Azumaya locus $\mathcal{PA}_R$ is an open subset of $\Maxspec(Z)$. Using similar techniques to those used in the proof of Theorem~\ref{PIDeg}, the proof of this fact when $R$ is prime (found, for example, in Proposition III.1.1 and Lemma III.1.5 in \cite{BGo2}) easily generalises to this case.

Note that the assumptions of Theorem~\ref{PIDeg} guarantee that $R$ is a Jacobson ring, i.e. that every prime ideal is an intersection of primitive ideals. In particular, $J(R)$ is the intersection of all prime ideals in $R$. Hence, if $R$ is a prime ring then $R$ has PI degree $d$ and the Azumaya and pseudo-Azumaya loci coincide by the following theorem (noting that, over a prime ring, if $R_\fm$ is an Azumaya algebra then it must be of constant rank as $Z(R_\fm)=Z_\fm$ is local for all maximal ideals $\fm$ of $Z$ -- see also Chapter 13.7 in \cite{McR}). Note that Brown and Goodearl have already shown the prime case in \cite{BGo}, using similar techniques.

\begin{theorem}
	Let $R$ be a $\bK$-algebra, where $\bK$ is an algebraically closed field, which is module-finite over its centre $Z=Z(R)$, and assume that $Z$ is affine. Suppose that $R$ has PI-degree $d$, where $d$ is the maximum dimension of an irreducible (left) $R$-module. Furthermore, let $M$ be an irreducible (left) $R$-module, $A=\ann_R(M)$ and $\fm=\ann_Z(M)$. Then $\dim(M)=d$ if and only if $R_\fm$ is an Azumaya algebra of constant rank $d^2$.
\end{theorem}

Note that, since $Z$ is affine, $\fm$ is a maximal ideal of $Z$.
%
%
%
%

\begin{proof}
	($\implies$)	Suppose that $M$ is an irreducible (left) $R$-module of dimension $d$. Then $R/A\cong M_d(\bK)$ and so $\PIDeg(R/A)=d=\PIDeg(R)$.
	
	In particular, this means that $A\in\Spec_d(R)$ and so $g_d(R)\not\subseteq A$. Thus, $g_d(R)\cap (Z\setminus \fm)\neq \emptyset$, and hence $g_d(R)$ contains an invertible element of $Z_\fm$, so an invertible element of $R_\fm$. Thus $g_d(R_\fm)\neq \{0\}$. Furthermore, any homogeneous multilinear polynomial identity of $R$ is a polynomial identity of $R_\fm$, and so $\PIDeg(R_\fm)=\PIDeg(R)$.
	
	Also, $1\in g_d(R_\fm)R_\fm$ since $g_d(R_\fm)$ contains an element of $Z\setminus \fm$. So by a version of the Artin-Procesi theorem (see \cite{Row}), $R_\fm$ is Azumaya over its centre of constant rank $d^2$.
	
	($\impliedby$) Suppose that $R_\fm$ is Azumaya of constant rank $d^2$ over its centre. In particular, the Artin-Procesi theorem from \cite{Row} tells us that $R_\fm$ has PI-degree $d$ and that $1\in g_d(R_\fm)R_\fm$. 
	
	Note that it is always true that $R/\fm R\cong R_\fm/\fm R_\fm$. Furthermore $R_\fm/\fm R_\fm$ satisfies all multilinear identities of $R_\fm$, and if $g_d(R_\fm)\subseteq \fm R_\fm$ then $1\in g_d(R_\fm)R_\fm\subseteq \fm R_\fm$. But then $\fm R_\fm=R_\fm$ which is a contradiction. So $R_\fm/\fm R_\fm$ has PI-degree $d$, and so $R/\fm R$ has PI-degree $d$. This precisely says that $\fm R\in\Spec_d(R)$, and so $\fm\in\Spec_d(Z)$.
	
	Since $\fm$ is a maximal ideal of $Z$, Theorem 1.9.21 of \cite{Row2} says that $\fm R$ is a maximal ideal of $R$, and so $A=\fm R$. In particular, $R/\fm R\cong M_d(\bK)$ as in the proof of Theorem~\ref{PIDeg}. Since $M$ is an irreducible $R/\fm R$-module, the result follows.
	%
	%
	%
	%
	%
	%
	%
	%
	%
	%
	%

\end{proof}
%

Observe that, by Schur's lemma, if $M$ is an irreducible $R$-module then each $u\in Z$ acts on $M$ by scalar multiplication. In particular, there exists a central character $\zeta_M:Z\to\bK$ where $\zeta_M(u)$ is defined by $u\cdot m=\zeta_M(u)m$ for all $m\in M$. Thus, 
$$\mathcal{PA}_R=\{\ker(\zeta_M)\,\vert\,M\,\mbox{an irreducible}\, R\mbox{-module of maximal dimension}\}.$$
%
%

\subsection{Pseudo-Azumaya loci for higher universal enveloping algebras}
\label{s3.2}

From now on, we once again suppose $\bK$ has characteristic $p>0$.
%
%

We now shall explore the pseudo-Azumaya locus for the higher universal enveloping algebras. Suppose that $G$ is a connected reductive algebraic group over $\bK$. We then take $Z_p^{[r]}$ to be the (central) subalgebra of $U^{[r]}(G)$ generated by the elements $\delta^{\otimes p}-\delta^{p}$ for $\delta\in\Di_{p^{r}}^{+}(G)$. The work of \cite{West} shows that $$Z^{[r]}_p=\bK[(\ve_{\alpha}^{(p^r)})^{\otimes p},\,\binom{\vh_t}{p^r}^{\otimes p}-\binom{\vh_t}{p^r}\,\vert\,\alpha\in\Phi,\,1\leq t\leq d].$$

Furthermore, from \cite{West} it is known that $U^{[r]}(G)$ is an affine $\bK$-algebra and that it is a free $Z_p^{[r]}$-module of finite rank $p^{(r+1)\dim(\fg)}$. Since $Z^{[r]}_p$ is Noetherian and finitely-generated, the Artin-Tate Lemma gives that the centre of $U^{[r]}(G)$, which we shall denote by $Z^{[r]}(G)$, is an affine $Z_p^{[r]}$-algebra and an affine $\bK$-algebra. This implies that $Z^{[r]}_p$, $Z^{[r]}(G)$ and $U^{[r]}(G)$ are Noetherian PI rings and that $U^{[r]}(G)$ is a Jacobson ring.  


For the remainder of this section we shall use the convention that for an irreducible $U(\fg)$-module $N$ the corresponding central character is $\zeta_N:Z(\fg)\coloneqq Z(U(\fg))\to\bK$ while for an irreducible $U^{[r]}(G)$-module $M$ the corresponding central character is $\zeta^{[r]}_M:Z^{[r]}(G)\to\bK$. In order to understand how these maps interact, we need to consider some homomorphisms between the centres.

Recall from \cite{West} that there exists a surjective algebra homomorphism $\Upsilon:U^{[r]}(G)\to U(\fg)$. This map clearly maps centres to centres, so gives an algebra homomorphism $\Upsilon\coloneqq\Upsilon_{r,r}:Z^{[r]}(G)\to Z(\fg)$. In particular, \cite{West} shows that, $\Upsilon((\ve_\alpha^{(p^r)})^{\otimes p})=\ve_\alpha^p$ for $\alpha\in\Phi$ and $\Upsilon(\binom{\vh_t}{p^r}^{\otimes p}-\binom{\vh_t}{p^r})=\vh_t^p-\vh_t$ for $1\leq t\leq d$. Hence, $\Upsilon$ further restricts to an algebra homomorphism 
$$\Upsilon:Z^{[r]}_p\to Z_p$$ 
which is now clearly an isomorphism.

There is another map between centres which is worth considering. Let $P$ be an irreducible $\Di(G_{r})$-module, and let us consider the induced module $U^{[r]}(G)\otimes_D P$, where, as always, $D$ denotes $\Di(G_{r})$. The action of $U^{[r]}(G)$ on $U^{[r]}(G)\otimes_D P$ is by left multiplication, so in particular $u\in Z^{[r]}(G)$ acts on $U^{[r]}(G)\otimes_D P$ by the $U^{[r]}(G)$-module endomorphism $$\rho(u):U^{[r]}(G)\otimes_D P\to U^{[r]}(G)\otimes_D P,$$ which is left multiplication by $u$. Clearly $\rho(u)$ is a central element of $E\coloneqq\End_{U^{[r]}(G)}(U^{[r]}(G)\otimes_D P)^{op}$. Recall from Proposition~\ref{End} that $U(\fg)$ is isomorphic to $E$, and let $\tau:E\to U(\fg)$ be the isomorphism. Hence, there is a homomorphism of algebras $$\Omega_P:Z^{[r]}(G)\to Z(\fg)$$ given by composition of $\tau$ and $\rho$.

We can furthermore observe that the proof of Proposition~\ref{Decomp} shows that $$\Omega_P((\ve_\alpha^{(p^r)})^{\otimes p})=\ve_\alpha^p$$ for $\alpha\in\Phi$ and $$\Omega_P(\binom{\vh_t}{p^r}^{\otimes p}-\binom{\vh_t}{p^r})=\vh_t^p-\vh_t$$ for $1\leq t\leq d$. In particular, $\Upsilon\vert_{Z^{[r]}_p}=\Omega_P\vert_{Z^{[r]}_p}$, and so $\Omega_P$ restricts to an isomorphism $Z_p^{[r]}\to Z_p$.

The following conditions for the map $\Omega_P$ to be surjective or injective are easy to prove,

\begin{lemma}
	The homomorphism $\Omega_P$ is surjective if and only if every central element of $E$ is left multiplication by some central element of $U^{[r]}(G)$.
\end{lemma}

\begin{lemma}
	The homomorphism $\Omega_P$ is injective if and only if, for $u\in Z^{[r]}(G)$, $u\otimes_D z=0\in U^{[r]}(G)\otimes_D P$ for all $z\in P$ implies that $u=0$. Equivalently, if and only if $U^{[r]}(G)\otimes_D P$ is a faithful $Z^{[r]}(G)$-module.
\end{lemma}

Let us see how the homomorphisms $\Omega_P$ interact with the central characters of irreducible $U^{[r]}(G)$-modules.

\begin{prop}
	Let $M$ be an irreducible $U^{[r]}(G)$-module with $\Psi(M)=(P,N)$ for $P\in \Irr(\Di(G_{r}))$ and $N\in\Irr(U(\fg))$. Then the following diagram commutes:
	$$
	\xymatrix{
		Z^{[r]}(G) \ar@{->}[rr]^{\zeta^{[r]}_M} \ar@{->}[d]^{\Omega_P} & & \bK \\
		Z(\fg) \ar@{->}[urr]^{\zeta_N} & & \\
	}
	$$
\end{prop}

\begin{proof}
	Recall here that $M\cong(U^{[r]}(G)\otimes_D P)\otimes_{U(\fg)} N$. Now, let $u\in Z^{[r]}(G)$, $v\in U^{[r]}(G)$, $z\in P$ and $n\in N$. Then \begin{multline*} 
		u\cdot (v\otimes_D z)\otimes_{U(\fg)} n=\rho(u)(v\otimes_D z)\otimes_{U(\fg)} n = (v\otimes_D z)\cdot\tau(\rho(u))\otimes_{U(\fg)} n \\ = (v\otimes_D z)\otimes_{U(\fg)} \Omega_P(u)\cdot n=\zeta_N(\Omega_P(u))(v\otimes_D z)\otimes_{U(\fg)} n
	\end{multline*}
\end{proof}

\begin{cor}
	Let $M$ be an irreducible $U^{[r]}(G)$-module with $\Psi(M)=(P,N)$ for $P\in \Irr(\Di(G_{r}))$ and $N\in\Irr(U(\fg))$. Then 
	$$\ker{\zeta^{[r]}_M}=\Omega_P^{-1}(\ker{\zeta_N}).$$
\end{cor}

Recall now from Corollary~\ref{maxdim} that if $M$ is an irreducible $U^{[r]}(G)$-module corresponding to the pair $(P,N)\in\Irr(\Di(G_{r}))\times\Irr(U(\fg))$ then $\dim(M)=\dim(P)\dim(N)$. Hence, an irreducible $U^{[r]}(G)$-module $M$ is of maximal dimension if and only if the corresponding modules $P$ and $N$ are of maximal dimension.

From now on fix $P$ as the $r$-th Steinberg module $St_r$ of $G$, hence an irreducible $\Di(G_{r})$-module of maximal dimension. As in Subsection~\ref{s2.1}, let $\Gamma_P$ be the category of irreducible $U^{[r]}(G)$-modules which contain $P$ as an irreducible $\Di(G_{r})$-submodule. Let $\Maxi\Gamma_P$ denote the full subcategory of $\Gamma_P$ whose objects are the irreducible $U^{[r]}(G)$-modules of maximal dimension in $\Gamma_P$, and let $\MaxIrr(U(\fg))$ similarly denote the full subcategory of $\Irr(U(\fg))$ consisting of irreducible $U(\fg)$-modules of maximal dimension. The inverse equivalences of categories $\Psi_P:\Gamma_P\to\Irr(U(\fg))$ and $\Phi_P:\Irr(U(\fg))\to\Gamma_P$ then restrict to inverse equivalences of categories $$\Psi_P:\Maxi\Gamma_P\to \MaxIrr(U(\fg))\qquad\mbox{and}\qquad \Phi_P:\MaxIrr(U(\fg))\to\Maxi\Gamma_P.$$

We have already seen that, for $M\in\Maxi\Gamma_P$, $\ker(\zeta^{[r]}_M)=\Omega_P^{-1}(\ker(\zeta_{\Psi_P(M)})$. We hence have that
\begin{multline*}
	\mathcal{PA}_{U^{[r]}(G)}=\{\ker(\zeta^{[r]}_M)\,\vert\,M\in\MaxIrr(U^{[r]}(G))\}=\{\ker(\zeta^{[r]}_M)\,\vert\,M\in\Maxi\Gamma_P\}\\=\{\Omega_P^{-1}(\ker(\zeta_{\Psi_P(M)}))\,\vert\,M\in\Maxi\Gamma_P\}=\{\Omega_P^{-1}(\ker(\zeta_{N}))\,\vert\,N\in\MaxIrr(U(\fg))\}.
\end{multline*}

\begin{prop}
	Let $P$ be the $r$-th Steinberg module $St_r$ of $G$. There is a surjective morphism $\Omega_P^{*}:\mathcal{PA}_{U(\fg)}\to\mathcal{PA}_{U^{[r]}(G)}$ which sends $\ker(\zeta_N)$ to $\Omega_P^{-1}(\ker(\zeta_N))$.
\end{prop}

\begin{proof}
	$\Omega_P:Z^{[r]}(G)\to Z(\fg)$ is a homomorphism of commutative algebras, so it induces a morphism
	$$\Omega_P^{*}:\Spec(Z(\fg))\to \Spec(Z^{[r]}(G)).$$
	This morphism sends $I\in\Spec(Z(\fg))$ to $\Omega_P^{-1}(I)\in\Spec(Z^{[r]}(G))$, so by above restricts to a map $\Omega_P^{*}:\mathcal{PA}_{U(\fg)}\to\mathcal{PA}_{U^{[r]}(G)}$. It is surjective by the above discussion.
\end{proof}

\begin{cor}
	Let $P$ be the $r$-th Steinberg module $St_r$ of $G$. If $\Omega_P$ is surjective, then $\Omega_P^{*}$ is a bijection.
\end{cor}

If we instead take $P$ to be an arbitrary irreducible $\Di(G_r)$-module then $\Psi_P$ and $\Phi_P$ still restrict to inverse equivalences of categories between $\Maxi\Gamma_P$ and $\MaxIrr(U(\fg))$, and we still get the equality $$\{\ker(\zeta^{[r]}_M)\,\vert\,M\in\Maxi\Gamma_P\}=\{\Omega_P^{-1}(\ker(\zeta_{N}))\,\vert\,N\in\MaxIrr(U(\fg))\},$$ but the left hand side may no longer be equal to $\mathcal{PA}_{U^{[r]}(G)}$. For example, if $P$ is the trivial 1-dimensional $\Di(G_r)$-module then $\Phi_P$ lifts an irreducible $U(\fg)$-module $N$ to the irreducible $U^{[r]}(G)$-module $N$ along the natural quotient $U^{[r]}(G)\mapsto U^{[r]}(G)/U^{[r]}(G)\Di^{+}(G_r)=U(\fg)$. Hence, if $N$ is an irreducible $U(\fg)$-module of maximum dimension, then $\ker(\zeta_N)$ is in the pseudo-Azumaya locus of $U(\fg)$ (and hence the Azumaya locus, since $U(\fg)$ is prime), but $\Omega_P^{*}(\ker(\zeta_N))=\ker(\zeta_N^{[r]})$. In particular, $\Omega_P^{*}(\ker(\zeta_N))$ will contain $Z\cap U^{[r]}(G)\Di^{+}(G_r)$, suggesting that it is not the central annihilator of an irreducible $U^{[r]}(G)$-module of maximum dimension.

\end{document}